\documentclass{amsart}

\usepackage[T1]{fontenc}
\usepackage[utf8]{inputenc}
\usepackage[english]{babel}

\usepackage{amsmath, amssymb, amsthm, amsfonts, mathrsfs}

\usepackage[letterpaper,top=2.5cm,bottom=2.5cm,left=3cm,right=3cm,marginparwidth=1.75cm]{geometry}
\usepackage{graphicx}
\usepackage{subcaption}
\usepackage{tikz}
\usepackage{color}

\usepackage{hyperref}
\usepackage{listings}
\usepackage{enumerate}
\usepackage{csquotes}
\usepackage{multirow}
\usepackage{cancel}

\usepackage{tikz-cd}

\newcommand{\Z}{\mathbb{Z}}
\newcommand{\Gal}{\mathrm{Gal}}
\newcommand{\Q}{\mathbb{Q}}
\newcommand{\Aut}{\mathrm{Aut}}
\newcommand{\Jac}{\mathrm{Jac}}

\newcommand{\et}{\mathrm{\acute{e}t}}

\newtheorem{theorem}{Theorem}
\newtheorem{lemma}[theorem]{Lemma}
\newtheorem{proposition}[theorem]{Proposition}
\newtheorem{corollary}[theorem]{Corollary}

\theoremstyle{definition} 

\newtheorem{example}[theorem]{Example}
\newtheorem{remark}[theorem]{Remark}

\usepackage{hyperref}
\usepackage{fancyhdr,textcase}
\title{On periods and Jacobians of Heisenberg curves}

\dedicatory{Dedicated to Professor Jannis A. Antoniadis}

\author{Dimitrios Noulas}

\address{Department of Mathematics, National and Kapodistrian University of Athens, Panepistimioupolis, 15784 Athens, Greece.}
\email{dnoulas@math.uoa.gr}

\begin{document}

\begin{abstract}
Heisenberg curves are cyclic covers of Fermat curves 
that also arise as non-abelian covers of the projective line, 
branched over three points by the 
discrete Heisenberg group modulo an integer. 
As normal Belyi covers, these are curves with many 
automorphisms in the sense of Oort, who questioned whether 
such curves have CM Jacobians. 
In 1986, Ihara proposed using towers of curves to 
study the pro-$\ell$ Galois representation associated 
with the thrice-punctured projective line. To study the kernel of 
this representation, he 
suggested using Heisenberg curves, but it was unknown to him 
at the time whether their Jacobians lacked complex multiplication. 
In this paper, for any odd prime $\ell$, 
we prove that Heisenberg curves of level $\ell^n\neq 3$ 
do not have CM Jacobians. Thus, we resolve the missing part 
of Ihara's original argument 
and in doing so we provide an infinite family of 
new counterexamples to Oort's question. 
\end{abstract}

\maketitle

{\flushleft
}

%11R58 	Arithmetic theory of algebraic function fields
%14H30 	Coverings of curves, fundamental group
%11R32 	Galois theory
%14E22 	Ramification problems in algebraic geometry

\section{Introduction}

The study of Heisenberg curves can be traced back to $1981$ in the works 
of Ramakrishnan and Murty \cite{RamakrishnanHeis},\cite{RamakrishnanMurtyHeis},
due to the monodromy of the dilogarithm function:
\[ \mathrm{Li}_2(z) = - \int_{0}^z \mathrm{log}(1-t) \frac{dt}{t}, \] yielding the following map
\[ \mathbb{P}^1_{\mathbb{C}}-\{0,1,\infty\} \longrightarrow  \left\{\begin{pmatrix}
    1 & * & * \\
    0 & 1 & * \\ 
    0 & 0 &  1
\end{pmatrix}, \quad * \in \mathbb{C} \right\} (\mathrm{ mod } M) , 
\quad  x \mapsto \begin{pmatrix} 
1 & - \mathrm{log}(1-x) & \mathrm{Li}_2 (x)\\
0 & 1 & \mathrm{log}(x)\\ 
0 & 0 & 1
\end{pmatrix}, \] where
\[ M = \left\{\begin{pmatrix}
    1 & a & c \\
    0 & 1 & b \\ 
    0 & 0 &  1
\end{pmatrix}, \quad a,b \in 2\pi i \Z, \ c \in (2\pi i)^2 \Z \right\}.\] Passing to the modular group $\Gamma(2)$, 
this map induces a homomorphism to the discrete Heisenberg 
group $H_\Z$ of $3\times 3$ unipotent upper triangular integer matrices, given on generators by: 
\[ \begin{pmatrix} 1 & 2 \\ 0 & 1  \end{pmatrix}, \begin{pmatrix} 1 & 0 \\ 2 & 1  \end{pmatrix} \longrightarrow 
 \begin{pmatrix} 1 & 1 & 0 \\ 0 & 1 & 0 \\ 0 & 0 & 1 \end{pmatrix}, 
 \begin{pmatrix} 1 & 0 & 0 \\ 0 & 1 & 1 \\ 0 & 0 & 1 \end{pmatrix},
\]

which fits the following commutative diagram:
\[
\begin{tikzcd}
\Gamma(2) \arrow[r] \arrow[rd] & H_\Z \arrow[d] \arrow[r, "\mathrm{ mod }n"] & H_n \arrow[d]        \\
                               & \Z\times \Z \arrow[r]                       & \Z/n\Z \times \Z/n\Z.
\end{tikzcd}
\] In this setting, let $y_n: \Gamma(2) \rightarrow (\Z/n\Z)^2$ be the natural projection. Then, the compactification 
of $\ker y_n \backslash \mathbb{H}$ is the Fermat curve of exponent $n$. Similarly, a Heisenberg curve arises 
as the compactification of $\ker x_n \backslash \mathbb{H}$ 
for the homomorphism $x_n: \Gamma(2)\rightarrow H_n$ defined above. For a general modular curve given by the compactification
 of 
$\Gamma \backslash \mathbb{H}$, where $\Gamma \subset \mathrm{SL}_2(\Z)$ is a finite-index subgroup, Ramakrishnan and 
Murty were interested in whether the images of the cusps in the associated Jacobian variety were torsion points, 
which property holds true for congruence subgroups by the Manin-Drinfeld theorem. Following Mazur's 
\cite[p. 39]{MazurModularEisenstein} analogy between 
Fermat curves and modular curves, Heisenberg curves can be viewed as covers of Fermat curves analogous to 
the Shimura coverings 
$X_1(N)\rightarrow X_0(N)$ \cite{Banerjee_Merel_2024} 
and can also provide a testing ground for the Manin--Drinfeld principle in the non-congruence setting. Indeed, 
it is proven in \cite{Banerjee_Merel_2024} that for $n=5$, a cusp point has infinite order in the Jacobian, causing the 
principle to fail.

As Heisenberg curves are normal Belyi curves, they play another role in arithmetic geometry regarding a question posed by Oort. 
In the moduli spaces $M_g$, $g\geq 2$, Oort defined a curve {\em with many automorphisms} as one whose corresponding 
point in the moduli space $M_g$ is a local maximum in terms of the size of the automorphism group. 
In fact, curves with many automorphisms correspond to 
normal Belyi curves of genus greater than $2$  (see Remark \ref{remark:many_aut}). 
Reflecting on the philosophy that curves with special geometric properties might also have special arithmetic features, 
Oort questioned whether curves with many automorphisms have Jacobian varieties with complex multiplication. 
The answer to this question turned out to be negative (see Section \ref{subsec:oort}), however, 
the counterexamples so far appeared of small bounded genus and did not seem to represent a general phenomenon.

The Jacobians of Heisenberg curves also attracted the interest of the Japanese mathematician Ihara, 
whose work inspired, and continues to inspire, the research in anabelian geometry both in Japan and worldwide. Section 
\ref{sec:ihara} 
is entirely devoted to his theory, its developments and detailed references.

Let $\ell$ be an odd prime number and 
$\mathcal{F}_2$ be the pro-$\ell$ completion of the free group on two generators 
$F_2 \cong \pi_1(\mathbb{P}^1-\{0,1,\infty\})$, which is realized as the topological fundamental 
group of the thrice-punctured projective line. In $1986$, Ihara initiated the study of the 
pro-$\ell$ Galois representation: %\svs 

\[ \phi: \operatorname{Gal}(\overline{\Q}/\Q) \longrightarrow \mathrm{Out}(\mathcal{F}_2), \] 
where $\mathrm{Out}(\mathcal{F}_2)$ denotes the outer automorphisms of $\mathcal{F}_2$. To obtain 
information about this homomorphism, its image and its kernel, Ihara used a combination of descending 
filtrations of $G_\Q:=\operatorname{Gal}(\overline{\Q}/\Q)$. One type of filtration arises by using the lower central series 
$\mathcal{F}_2(m+1) = [\mathcal{F}_2, \mathcal{F}_2(m)]$ to replace $\mathcal{F}_2$ with the quotient 
$\mathcal{F}_2/\mathcal{F}_2(m+1)$ in $\phi$. This yields the homomorphisms $\phi_m :G_\Q \rightarrow \mathrm{Out}(
  \mathcal{F}_2/\mathcal{F}_2(m+1))$, and the kernels provide the filtration.

The second type of filtration relies on
well-chosen infinite-index subgroups $\mathcal{N}$ of $\mathcal{F}_2$, 
to replace $\mathcal{F}_2$ with $\mathcal{F}_2^* := \mathcal{F}_2/[\mathcal{N},\mathcal{N}]$ and then, as 
previously, to replace $\mathcal{F}_2^*$ with $\mathcal{F}_2^*/\mathcal{F}_2^*(m+1)$, yielding the homomorphisms $\psi_m :G_\Q \rightarrow \mathrm{Out}(
  \mathcal{F}_2^*/\mathcal{F}_2^*(m+1))$ along with their respective kernels. 

The groups $\mathcal{N}$ play the role of {\em towers of curves} in the following sense. 
If $x,y,z$ are the generators of homotopy classes of 
loops around $0,1,\infty$, consider the group 
$\mathcal{N}_{\ell^n} = \mathcal{N}\cdot \langle\langle x^{\ell^{n}},y^{\ell^n},z^{\ell^n} \rangle\rangle$, 
where $\langle\langle . \rangle\rangle$ is the normal and Krull closure in the pro-$\ell$ topology. Then, if $\mathcal{N} = \mathcal{F}_2(2)$ (resp. $\mathcal{F}_2(3)$) 
the group $\mathcal{N}_{\ell^n}$ corresponds 
to a Fermat (resp. Heisenberg) curve of level $\ell^n$, 
and $\mathcal{N}/[\mathcal{N},\mathcal{N}]$ is canonically isomorphic to the 
inverse limit of the $\ell$-adic Tate modules of their respective Jacobians.

A year later in $1987$, before explicitly determining $\ker \phi$ in a collaboration with Anderson \cite{AndersonIhara88}, 
Ihara hinted that with the above machinery on Heisenberg curves, one could possibly deduce 
that the fixed field of $\ker\phi$ is a non-abelian extension of $\Q(\mu_{\ell^\infty})$, where 
$\mu_{\ell^\infty}$ is the group of all $\ell$-power roots of unity. However, he wrote that at the time he did not know
whether the Jacobians of Heisenberg curves for $\ell >3 $ lacked complex multiplication \cite{Ihara3point}.

Our main contribution in this paper is the following theorem, which verifies Ihara's intuition. %\svs

\begin{theorem}\label{thm:intro}
  Fix an odd prime $\ell$ and let $X_{\ell^n}$ be the Heisenberg curve over $\mathbb{C}$ of level $\ell^n$. If $\ell>3$, then the Jacobian 
  $\Jac(X_{\ell^n})$ does not have complex multiplication. If $\ell=3$, the same holds for $\Jac(X_{3^n})$ when $n\geq 2$.
\end{theorem}

The $\ell>3$ and $\ell=3$ cases are treated separetely, and Theorem \ref{thm:intro} 
is the combination of Theorems \ref{thm:main} and \ref{thm:complement}. In proving Ihara's statement, 
we simultaneously provide an infinite family of new counterexamples to the question by Oort about 
curves with many automorphisms. Furthermore, the genus of these counterexamples grows arbitrarily large 
(see the odd part of Equation (\ref{eq:genus})).

We prove Theorem \ref{thm:intro} by combining the theory of Chevalley-Weil \cite{Chevalley1934-eb} regarding the action of the 
automorphism group on the space of holomorphic differentials with the work 
of Wolfart \cite{WolfartCMJacobians}, which specializes this setting to Belyi curves and relates
 the semisimple decomposition of the representations with the isogeny
decomposition of the Jacobian varieties. In doing so, for $\ell>3$ we provide the isogeny decomposition 
(see Corollary \ref{cor:matching_decomps}):
\[ \Jac(X_{\ell}) \sim \Jac(Y_{\ell}) \times \mathcal{A}_1^{k_1} \times \mathcal{A}_2^{k_2} \times \cdots \times \mathcal{A}_{\ell-1}^{k_{\ell-1}}, \]
where $Y_\ell$ is the Fermat curve of level $\ell$ and the $\mathcal{A}_i$ are simple abelian varieties. We then apply
Wolfart's arguments to our study of the space of holomorphic differentials, as well as 
the classification of endomorphism algebras of simple abelian varieties over $\mathbb{C}$ attributed to Albert and Shimura 
\cite{AlbertA},\cite{MR156001}, 
to show that if any $\mathcal{A}_j$ is assumed to have complex multiplication, it leads to a 
contradiction regarding its dimension over $\mathbb{C}$. The full theorem then follows by lifting this 
base case to the entire $\ell$-tower of curves using standard properties of abelian varieties. 
The $\ell=3$ case follows by adjusting these arguments specifically 
for the $X_9$ Heisenberg curve.

We note that the $\ell=2$ case is not considered; a Heisenberg curve of even level $n$ is 
not unique up to isomorphism, and one needs to keep track of the ``coordinate system'' $x,y,z$ of the homotopy 
classes of loops around $0,1,\infty$. This does not happen in the odd case, as a consequence of the classification of the 
automorphism group of Heisenberg curves by Antoniadis and Kontogeorgis \cite{MR4252293}. Even if one keeps 
track of the coordinate system, the discrete Heisenberg group modulo $n$ has a fundamentally different group structure 
in the odd versus even cases. Accommodating this would require substantial modifications throughout the present paper. 
For instance, in the even case Heisenberg curves are no longer unramified covers
of Fermat curves and consequently the previous statement about $\mathcal{N}_{\ell^n}$ is no longer true, 
meaning this group for $\ell=2$ does not 
correspond to a Heisenberg curve.   
Therefore, to avoid these complications, we restrict our focus to the odd primes $\ell$. 

Our second contribution combines a discrete version of the tangential basepoint idea introduced by Deligne with the Fox derivatives--a tool 
Ihara heavily relied on--in order to recompute the Fermat periods as in the appendix by Rohrlich from the 
article \cite{GrossAbelianIntegrals} by Gross. 
Having also constructed a basis of holomorphic differentials, we then generalize this approach 
to compute the periods of Heisenberg curves $X_{\ell^n}$. This method 
offers the advantage of eliminating the need for explicit contour integration and manual tracking of winding numbers. 
Instead, the Fox derivatives and the curve automorphisms perform this task automatically.
Furthemore, transcendence 
results regarding the period matrices are discussed in Remark \ref{remark:periods}, as a consequence of Theorem \ref{thm:intro}. 

Regarding other recent work on Heisenberg curves, the dilogarithm viewpoint has been revisited in \cite{MR3912937}. 
Some parts of our representation-theoretic arguments, 
mainly concerning the automorphism module structure of the homology groups $H_1(X_{\ell^n},\mathbb{C})$, 
rely on previous work of the author in 
\cite{kontogeorgis2024galoisactionhomologyheisenberg}, which we recover here through different means. Finally, 
for more work on 
Heisenberg function field extensions, we refer to \cite{MR4083092}, which builds upon \cite{sharifi1999twisted}.

The paper is organized as follows. In Section \ref{sec:heis_curves} we formally define Heisenberg curves 
and provide an affine model, followed in \ref{subsec:Hn_reps} 
by equations for the characters of the irreducible representations of the Heisenberg 
group $H_n$. In Section \ref{sec:holo}, we study the space of 
holomorphic differentials and its automorphism module structure, providing a basis in \ref{subsec:basis}. 
In Section \ref{sec:periods} we employ Fox derivatives and the tangential basepoint to 
compute the periods. In Section \ref{sec:rep_theory}, we study the 
$H_{\ell^n}$-module structure of $H^0(X_{\ell^n}, \Omega_{\mathbb{C}})$,  
$H^1(X_{\ell^n},\mathbb{C})$ and $H_1(X_{\ell^n},\mathbb{C})$, 
via the Chevalley--Weil theorem. In Section \ref{sec:decomp}, we combine the findings of Section \ref{sec:rep_theory} 
with Wolfart's framework to decompose the Jacobians $\Jac(X_{\ell^n})$ into powers of simple abelian varieties. 
We also provide an alternative decomposition into Jacobians of subcovers 
in \ref{subsec:KaniRosen} using a theorem of Kani and Rosen. In Section \ref{sec:cm}, we prove the $\ell>3$ part of 
Theorem \ref{thm:intro} and prove the $\ell=3$ part in \ref{subsec:ell_three}. Then, we discuss Oort's 
question and previous known counterexamples in \ref{subsec:oort}. In \ref{subsec:streit} we make remarks 
on the possible converse of Streit's CM criterion. Finally, in Section \ref{sec:ihara}, we give a brief review of selected aspects of 
Ihara theory, 
related recent developments and his initial motivation for the question regarding 
Heisenberg Jacobians.
\section*{Acknowledgments}
I am deeply grateful to my advisor, Aristides Kontogeorgis, for 
his constant support and guidance throughout my doctoral studies.
Many thanks also to Ilias Andreou, Miltiadis Karakikes, Orestis Lygdas, 
Konstantia Manousou and Alexios Terezakis for our inspiring conversations 
at the University of Athens. I would also like to thank my collaborators 
Nikita Andrusov, Sevag Büyüksimkeşyan, Fabien Pazuki, 
Mustafa Umut Kazancıoğlu and Jordi Vilà-Casadevall from the paper: 
``Distortion maps for elliptic curves over finite fields'', as that project significantly 
deepened my understanding of complex multiplication theory, leading to the present work.

\section{Heisenberg Curves}\label{sec:heis_curves}

For a positive integer $n$, the term {\em Heisenberg curve} $X_n$ of level $n$ will denote a non-singular projective curve over $\overline{\Q}$ such that, as a normal branched cover of the projective line, the Galois group $\Gal( X_n/\mathbb{P}_{\overline{\Q}}^1)$ is isomorphic to the following group:
\[H_n := \left\{\begin{pmatrix}
    1 & * & * \\
    0 & 1 & * \\ 
    0 & 0 &  1
\end{pmatrix}, \quad * \in \Z/n\Z \right\}.\] The group $H_n$ is referred to as the (discrete) {\em Heisenberg} group 
modulo $n$. By the Riemann Existence theorem, finite branched covers of the Riemann sphere $\mathbb{P}^1_\mathbb C$ correspond to complex algebraic curves. 
The existence of such a curve over $\overline{\Q}$ is then assured by the theorem of 
Belyi \cite{Belyi1}, which asserts that an algebraic curve $Y$ over 
$\mathbb{C}$ has a $\overline{\Q}$-model, if and only if there exists a 
morphism of curves $Y\rightarrow \mathbb{P}^1$ that is branched at most 
three points of $\mathbb{P}^1$. By composing with a Möbius transformation, the 
three points can be assumed to be $0,1,\infty$. The topological fundamental group $\pi_1(\mathbb{P}_{\mathbb{C}}^1-\{0,1,\infty\})$ is free on two generators, as it admits 
a presentation $\langle x,y,z \mid xyz=1\rangle $, where $x,y,z$ represent homotopy classes of loops around the 
branch points $0,1,\infty$. As the group $H_n$ is generated by the two elements
\[\alpha := \begin{pmatrix}
  1 & 1 & 0\\
  0 & 1 & 0 \\
  0& 0 & 1\end{pmatrix}, \quad \beta := \begin{pmatrix}
  1 & 0 & 0\\
  0& 1 & 1 \\
  0& 0& 1\end{pmatrix},\] the kernel of any surjective homomorphism
  \( \pi_1(\mathbb{P}_{\mathbb{C}}^1-\{0,1,\infty\}) \longrightarrow H_n,\) corresponds to a topological cover of the open curves $X_n^\circ(\mathbb{C})\rightarrow \mathbb{P}_{\mathbb{C}}^1-\{0,1,\infty\}$, where $X_n^\circ$ is a Heisenberg curve $X_n$ minus the ramified points, which depends on the choice of the above homomorphism.

  In terms of uniformization, the projective line $\mathbb{P}^1$ with three punctures admits the 
  hyperbolic space $\mathbb{H}$ as its universal cover and \[\Gamma(2)\setminus 
  \mathbb{H} \cong \mathbb{P}^1_\mathbb{C} - \{0,1,\infty\},\]
where \[\Gamma (2) = \left\{ \gamma = \begin{pmatrix}
  a & b \\ 
  c & d 
\end{pmatrix} \in \mathrm{SL}_2(\Z): \ \gamma \equiv \begin{pmatrix} 1 & 0 \\ 0 & 1 
\end{pmatrix} \mod 2 \right\} = \left\langle \begin{pmatrix} 1 & 2 \\ 0 & 1 
\end{pmatrix}, \begin{pmatrix} 1 & 0 \\ 2 & 1 
\end{pmatrix} \right\rangle. \] Furthermore, $\Gamma(2)$ is also freely generated by the two matrices above. Thus any surjective homomorphism $\psi : \Gamma(2) \rightarrow H_n$ yields an open Heisenberg curve as the quotient $X^\circ_n = \ker \psi \setminus \mathbb{H}$. 

To obtain the compactified Riemann surfaces, we uniformize via finite-index subgroups of triangle groups, which are defined as follows.
\[\Delta(k_0,k_1,k_\infty) := \langle \gamma_0,\gamma_1,\gamma_\infty \mid \gamma_0^{k_0} = \gamma_1^{k_1} = \gamma_\infty^{k_\infty} = 
\gamma_0\gamma_1\gamma_\infty = 1 \rangle. \] By considering parabolic elements in $\Gamma(2)$ as elliptic elements of infinite period, we can realize 
$\Gamma(2)$ as $\Delta(\infty,\infty,\infty)$. Thus the quotients $\Delta(k_0,k_1,k_\infty)$ act on $\mathbb{H}$ by fractional linear transformations. 
In particular, for finite $k_i$, the generators $\gamma_i$ act on $\mathbb{H}$ as rotations with angles $2\pi/k_i, i=0,1,\infty$ in the positive sense around their respective fixed points. 
Let $\psi_n$ be a surjective homomorphism
$\Delta(n,n,n) \rightarrow H_n$ if $n$ is odd and $\Delta(n,n,2n)\rightarrow H_n$ 
if $n$ is even. Then $X_n:= \ker \psi_n \setminus \mathbb{H}$ is a Heisenberg 
curve, dependent on $\psi_n$. The reason for this distinction is that as 
$H_n$ is $2$-nilpotent, it holds that $(\alpha \beta)^m = 
\alpha^m \beta^m [\beta,\alpha]^{\binom{m}{2}}$. The center $Z(H_n)$ 
is of order $n$ generated by $c:=[\alpha,\beta] = 
\alpha \beta \alpha^{-1}\beta^{-1}$, thus $(\alpha\beta)$ is of order 
$n$ (resp. $2n$) if and only if $n$ is odd (resp. even). Of course, in the even case similar choices can be made for $\Delta(2n,n,n)$ and $\Delta(n,2n,n)$. 
By the Riemann-Hurwitz formula, the genus $g_n$ of $X_n$ can be calculated as follows.
\[ 2 g_n - 2 = |H_n| \left(1- \frac{1}{k_0} - \frac{1}{k_1} - \frac{1}{k_\infty} \right), \] which yields:
\begin{equation}\label{eq:genus}
g_n = \begin{cases} n^2(n-3)/2 + 1, & \textrm{ if } (n,2)=1, \\
  n^2(n-3)/2 + n^2/4 + 1, & \textrm{ otherwise.}
\end{cases}
\end{equation} See also \cite[Lemma 15]{MR4252293}.

We will restrict to the odd case, as we 
will be interested in Heisenberg curves of level a power of an odd prime 
$\ell$. In the odd case, Antoniadis and Kontogeorgis in \cite[Lemma 16, Thm. 18]{MR4252293} 
prove that the group of automorphism $\Aut (X_n)$ is isomorphic to 
$H_n \rtimes S_3$\footnote{They do prove that every automorphism of 
$X_n$ is modular. Then the claim follows by lifting the $S_3$ 
automorphisms from the Fermat curves and thus the exact sequence in Thm. 
18 is split. It should be noted that the presentation of the kernel $\ker\phi$ in \S 6.2 is 
not precise and should be replaced by 
$\langle a^n, b^n, [a,[a,b]],[b,[a,b]]\rangle$. Then it can be shown that for odd $n$ 
this group is 
invariant by 
the $i_1,i_2$ involutions of $F_2$ generating $S_3$.}. This implies the action of 
$S_3$ on $\mathbb{P}^{1}$ 
as the ''$\lambda$-group'' 
$\{\lambda, 1-\lambda, \frac{1}{\lambda},\frac{1}{1-\lambda}, 
\frac{\lambda}{\lambda-1},\frac{\lambda-1}{\lambda}\}$--which permutes $0,1,\infty$--lifts to an isomorphism of $X_n$. Therefore, any choice of a homomorphism 
$\psi_n: \Delta(n,n,n)\rightarrow H_n, \gamma_i \mapsto \alpha,\beta$ gives the same curve up to isomorphism over $\overline{\Q}$ or $\mathbb{C}$.

We shall thus work with a specific model of level $\ell^n$ of our choice and as long as its Galois group 
over $\mathbb{P}^1$ is isomorphic to $H_{\ell^n}$, we can disregard the choice of the homomorphism 
$\psi_{\ell^n}$ and refer to it as {\em the Heisenberg curve $X_{\ell^n}$}. 
For our equations, we will follow the construction of Hirano and Morishita in 
\cite[Thm. 2.1.4]{MR3912937}, which readily extends from level $\ell$ to level $\ell^n$. See also \cite[Thm. 1]{Banerjee_Merel_2024}.
Let $K=\overline{\Q}(t)$ be the projective $t$-line $\mathbb{P}^1$ over $\overline{\Q}$. We denote by 
$\mathcal{K}_{\ell^n}$ the extension $K(t^{1/\ell^n}, (1-t)^{1/\ell^n})$. The extension $\mathcal{K}_{\ell^n}/K$ is
 Galois and the group $\Gal(\mathcal{K}_{\ell^n}/K)$ is isomorphic to $\Z/\ell^n\Z \times \Z/\ell^n \Z$,
  generated by $\tilde{\alpha},\tilde{\beta}$ defined as follows. Fix $\zeta$ a primitive $\ell^n$-th root of unity, then
\begin{align*} \tilde{\alpha}(t^{1/\ell^n}) &:= \zeta t^{1/\ell^n}, \quad \tilde{\alpha}((1-t)^{1/\ell^n}) =(1-t)^{1/\ell^n}, \\ 
\tilde{\beta}(t^{1/\ell^n}) &:= t^{1/\ell^n}, \quad \tilde{\alpha}((1-t)^{1/\ell^n}) = \zeta (1-t)^{1/\ell^n},
\end{align*}
and is unramified outside $0,1,\infty$. The ramification indices of these points are $\ell^n$. With the function 
field $\mathcal{K}_{\ell^n}$ is associated a non-singular projective model, the Fermat curve $Y_n$ of level $\ell^n$ 
\[ X^{\ell^n}+Y^{\ell^n} = Z^{\ell^n} \textrm{ in } \mathbb{P}^2,\] and the Belyi map 
$Y_n\rightarrow \mathbb{P}^1$ is given by $(X:Y:Z)\mapsto (X^{\ell^n},Z^{\ell^n})$. Set $x=t^{1/\ell^n},y=(1-t)^{1/\ell^n}$ and $\varepsilon = \varepsilon(t)$ such that
\begin{equation}\label{eq:varepsilon_ln}
\varepsilon^{\ell^n} := \prod\limits_{i=1}^{\ell^n-1}(1-\zeta^i x)^i.
\end{equation} Define the extension $\mathcal{R}_{\ell^n} := \mathcal{K}_{\ell^n}(\varepsilon) = \overline{\Q}(x,y,\varepsilon)$. The 
extension $\mathcal{R}_{\ell^n}/\mathcal{K}_{\ell^n}$ is a cyclic Kummer extension of degree $\ell^n$ with Galois group 
$\Gal(\mathcal{R}_{\ell^n}/\mathcal{K}_{\ell^n})$ generated by $c$ defined as $c(\varepsilon) := \zeta \varepsilon$. As we have 
already defined $c$ as an element in $H_{\ell^n}$, we shall show the definitions coincide by showing that 
$\Gal(\mathcal{R}_{\ell^n}/K)$ is isomorphic to $H_{\ell^n}$, which also implies that $\mathcal{R}_{\ell^n}/\mathcal{K}_{\ell^n}$ 
is unramified for our setting with odd primes $\ell$. Firstly, 
to discuss the level $\ell^n=3$ case, in \cite[Example 2.1.3]{MR3912937} a non-singular projective model is given, which is the 
elliptic curve $\zeta^2 X^3 + Y^3 = -\zeta W^3$. Furthermore, in \cite{MR4252293} a different (isomorphic) model is given, the elliptic curve with 
affine equation $Y^2 = X^3+2^3\cdot 3^6$. The two equations represent the same elliptic curve over $\overline{\Q}$ as they both have $j$-invariant $0$
and this curve admits complex multiplication by $(X,Y)\mapsto(\zeta X,Y)$, a fact which will be relevant later.

\begin{proposition}\label{prop:Gal_H_n}
With notations as above, the Galois group $\Gal(\mathcal{R}_{\ell^n}/K)$ is isomorphic to $H_{\ell^n}$, and the extension 
$\mathcal{R}_{\ell^n}/K$ is unramified outside $0,1,\infty$.
\end{proposition}
\begin{proof}
The proof will be similar to the proof of \cite[Thm. 2.1.4]{MR3912937} but for level $\ell^n$. The claim about the ramification follows from the intermediate extensions $\mathcal{R}_{\ell^n}/\mathcal{K}_{\ell^n}$ and $\mathcal{K}_{\ell^n}/K$. 
For $0\leq j < \ell^n$, we have that 
\[ \tilde{\alpha}^j (\varepsilon^{\ell^n}) = \frac{\{ (1-x)\cdots(1-\zeta^{j-1}x)\}^{\ell^n}}{(1-x^{\ell^n})^j} 
\varepsilon^{\ell^n}, \quad \tilde{\beta}^j (\varepsilon^{\ell^n}) = \varepsilon^{\ell^n}.\] Define (again) as 
$\alpha,\beta$ the extensions of $\tilde{\alpha},\tilde{\beta}$ from $\Gal(\mathcal{K}_{\ell^n}/K)$ to $\Gal(\mathcal{R}_{\ell^n}/K)$ by
\begin{align*}
\alpha(x):=\zeta x, \quad \alpha(y):= y,  \quad &\alpha(\varepsilon):= \frac{1-x}{(1-x^{\ell^n})^{1/\ell^n}} \varepsilon = \frac{1-x}{y} \varepsilon, \\
\beta(x):=x, \quad \beta(y) = \zeta y, \quad & \beta(\varepsilon) := \varepsilon.
\end{align*} It is easy to verify that $\beta^\ell=1$. The other case $\alpha^{\ell^n}=1$ happens precisely because $\ell^n$ is odd. Indeed, we need to verify it holds for 
$\alpha^{\ell^n}(\varepsilon)$, which is 
as follows,
\[\alpha^{\ell^n}(\varepsilon) = \frac{1}{1-x^{\ell^n}} \cdot (1-x)(1-\zeta x)\cdots(1-\zeta^{\ell^n-1}x) 
\varepsilon = \zeta^{\binom{\ell^n}{2}} \varepsilon = \varepsilon.\] By computation, we have $[\alpha,\beta](x) = x, 
[\alpha,\beta](y)=y$ and $[\alpha,\beta](\varepsilon) = \zeta \varepsilon$, thus $[\alpha,\beta]=c$. 
Therefore, there is an isomorphism $\Gal(\mathcal{R}_{\ell^n}/K) \rightarrow H_{\ell^n}$ mapping
\[\alpha \mapsto \begin{pmatrix}
  1 & 1 & 0\\
  0& 1 & 0 \\
  0& 0& 1\end{pmatrix}, \quad \beta \mapsto \begin{pmatrix}
  1 & 0 & 0\\
  0& 1 & 1 \\
  0& 0& 1\end{pmatrix},\] and thus our definitions of $\alpha,\beta,c=[\alpha,\beta]$ in $H_{\ell^n}$ and $\Gal(\mathcal{R}_{\ell^n}/K)$ coincide. 
\end{proof}
By our discussion of uniqueness of the $\ell^n$-level Heisenberg curve $X_{\ell^n}$, Proposition \ref{prop:Gal_H_n} 
now asserts that the system $x^{\ell^n}+y^{\ell^n}=1$ along with the defining Equation (\ref{eq:varepsilon_ln}) 
of $\varepsilon^{\ell^n}$ is an affine model for $X_{\ell^n}$ which is what we will work with in the forthcoming sections. 
We close the section with the following remark about curves with many automorphisms.

\begin{remark}\label{remark:many_aut}
An algebraic curve $\mathcal{C}$ over $\mathbb{C}$ of genus $g\geq 2$ is said to have {\em many automorphisms} if its 
corresponding point $\mathfrak{p}$ in the moduli space $\mathcal{M}_g$ has a neighborhood $V$ 
(in the complex topology) such that all curves corresponding to any other point in $V\setminus\{\mathfrak{p}\}$ 
have strictly smaller automorpshism group than $\mathrm{Aut}(\mathcal{C})$. These are not to be confused 
with curves with {\em large automorphism group} $G$, which are curves with $|G|>4(g-1)$. 
There are curves with large automorphism group that do not satisfy the many automorphisms condition. A 
reformulation of the many automorpshisms condition is that $\mathcal{C}$ can be uniformized by a finite index torsion 
free subgroup $\Gamma$ of a hyperbolic triangle group $\Delta(k_0,k_1,k_\infty)$. 
For the equivalence, see Wolfart's \cite[Thm. 6, Lemma 8]{WolfarObviousBelyi}.
As the $\ell^n$-th level Heisenberg 
curve $X_{\ell^n}$ is uniformized by the third term of the lower central series of $\Delta(\ell^n,\ell^n,\ell^n)$, this 
implies each $X_{\ell^n}$ for every odd prime $\ell$ (except for the level $\ell^n=3$ case of genus $1$) is a curve 
with many automorphisms. 

\end{remark}
\subsection{Representation theory of \texorpdfstring{$H_n$}{Hn}} \label{subsec:Hn_reps}
We shall briefly review at this point the decomposition of the regular representation $\mathbb{C}[H_n]$ into 
indecomposables, for all positive integers $n$. The irreducible characters of $H_n$ and the 
indecomposables will become important later when we discuss the Chevalley-Weil theorem as well as the 
decomposition of Jacobians of the Heisenberg curves.

As $H_n$ can be realized as $(\Z/n\Z \times \Z/n\Z) \rtimes \Z/n\Z$, the standard reference is Serre's book \cite[\S 8.2]{SerreLinear} for semi-direct 
products with abelian kernel. The author has performed this procedure in previous work and provides a relevant 
appendix in \cite{kontogeorgis2024galoisactionhomologyheisenberg}, which we will borrow the equations from.

For $j=0,1,\ldots,n-1$ set $d_j = \gcd(n,j)$, then the characters of irreducible representations of $H_n$ are 
denoted by $\chi_{ijs}$, for $i,s=0,1,\ldots,d_j-1$, and are defined as follows. 
\begin{equation}\label{eq:chi_ijs}
\chi_{ijs}( \beta^m c^\lambda \alpha^\mu) = \frac{n}{d_j} \zeta^{ms + \lambda j + \mu i}, \ 
\textrm{ if } {n}/{d_j} \textrm{ divides } \mu \textrm{ and } m,\end{equation} and $0$ otherwise, where $\zeta$ is a primitive $n$-th root of unity and
\[\beta^m \alpha^\lambda c^\mu = \begin{pmatrix}
1 & \lambda & \mu \\
0 & 1 & m \\
0 & 0 & 1
\end{pmatrix} \in H_n.\]
Each character $\chi_{ijs}$ is of dimension $n/d_j$. Thus, to put everything together, the character $\chi_{\mathbb{C}[H_n]}$ of $\mathbb{C}[H_n]$ is the following sum. 
\[\chi_{\mathbb{C}[H_n]} = \sum\limits_{j=0}^{n-1}\sum\limits_{i,s=0}^{d_j-1} \frac{n}{d_j} \chi_{ijs},\]
and the same equation holds over $\overline{\Q}$ or over any field satisfying Maschke's theorem for $H_n$.

\section{The space of holomorphic differentials}\label{sec:holo}

\noindent 
In this section we shall study the space of holomorphic differentials $H^0(X_n, \Omega_{\overline{\Q}})$ as a $\overline{\Q}$-vector space, 
where $n$ is a power of an odd prime, $X_n$ is the $n$-level Heisenberg curve and $\Omega_{\overline{\Q}}$ is the sheaf of 
holomorphic 
differentials defined over $\overline{\Q}$ associated to the curve. The purpose of this is two-fold: first, to explicitly describe the periods of $X_n$. Second, to study it as a 
$\overline{\Q}[H_n]$-module, which will lead to the decomposition of the Jacobian $\mathrm{Jac}(X_n)$.

Recall that we are working with the following affine model for $X_n$ and its function field is denoted by $\mathcal{R}_n$:
\[x^n+y^n=1, \quad  \varepsilon^n = \prod\limits_{i=1}^{n-1}(1-\zeta^ix)^i,\] where $\zeta$ is a primitive $n$-th root of unity.

Based on the defining affine equation $x^n+y^n=1$, we can utilize the well-understood theory of divisors and holomorphic differentials on the $n$-level Fermat curves $Y_n$, 
with function field $\mathcal{K}_n$. We mention the divisors of $\mathcal{K}_n$ we require and for detailed 
computations on them we will 
refer to \cite{Ko:98}.
For the entire section we set $\omega := dx/y^{n-1} = - dy/x^{n-1}$. Then $(\omega)_{\mathcal{K}_n}$ is a canonical divisor of 
$\mathcal{K}_n$. Furthermore, we denote by $\alpha_i,\gamma_i$ the places of $\mathcal{K}_n$ lying above the places $P_{(x=0)}$ and $P_{(x=\infty)}$ respectively, in $\overline{\Q}(x)$. By symmetry, 
the places $\gamma_i$ coincide also with the places above $P_{(y=\infty)}$ in $\overline{\Q}(y)$. Denote also by 
$\beta_i:=P_{(x=\zeta^{-i},y=0)}$ the places above $P_{(y=0)}$. We have the following divisors of the $n$-level Fermat curve.
\[ (x)_{\mathcal{K}_n} = \sum\limits_{i=1}^n \alpha_i - \sum\limits_{i=1}^n \gamma_i, \quad 
(y)_{\mathcal{K}_n} = \sum\limits_{i=1}^n \beta_i - \sum\limits_{i=1}^n \gamma_i,\] \[(dx)_{\mathcal{K}_n} = (n-1)\sum\limits_{i=1}^n \beta_i - 2\sum\limits_{i=1}^n \gamma_i, \quad 
(\omega)_{\mathcal{K}_n} = \frac{2g_{\mathcal{K}_n} - 2}{n} \sum\limits_{i=1}^n \gamma_i, \quad 
(1-\zeta^ix)_{\mathcal{K}_n} = n \beta_i - \sum\limits_{i=1}^n \gamma_i,\]
where $g_{\mathcal{K}_n}=(n-1)(n-2)/2$ is the genus of $Y_n$. Set $D=\frac{1}{n} (\varepsilon^n)_{\mathcal{K}_n}$, which is a divisor of 
$\mathcal{K}_n$. 
Based on the above divisors, we can compute $D$ as follows.
\begin{align*} D &= \frac 1n \left( \prod\limits_{i=1}^{n-1} \left(1-\zeta^i x\right)^i \right)_{\mathcal{K}_n}= \frac 1n \sum\limits_{i=1}^{n-1} i\left(1-\zeta^i x\right)_{\mathcal{K}_n}  = \frac 1n \sum\limits_{i=1}^{n-1} n\cdot i\cdot  \beta_i - \frac{n(n-1)}{2n}\sum\limits_{i=1}^n \gamma_i \\
& = \sum\limits_{i=1}^{n-1} i\cdot  \beta_i - \frac{(n-1)}{2}\sum\limits_{i=1}^n \gamma_i =: D_0 - D_\infty,
\end{align*} using the fact that $n$ is odd. Observe that the place $\beta_n = P_{(x=1,y=0)}$, does not appear in $D_0$ which we will have to take into account. Furthermore, for $j=1,\ldots,n-1$ we write the following equation of $jD$ for future reference. 
\begin{equation} \label{eq:jD}
jD = \sum_{i=1}^{n-1}ij \cdot \beta_i - \frac{j(n-1)}2 \sum\limits_{i=1}^n \gamma_i. 
\end{equation} We shall now calculate the dimensions of the Riemann-Roch spaces $L((\omega)+jD)$, where 
$L(E)=\{ f \in \mathcal{K}_n: (f)_{\mathcal{K}_n}+E\geq 0\}$, for any divisor $E$ of $\mathcal{K}_n$. As customary, we denote the dimension of $L(E)$ by $l(E)$. Set 
$E_j := (\omega)_{\mathcal{K}_n}+jD$, for $1\leq j \leq n-1$. Then, by the Riemann-Roch theorem we have that
\[l(E_j)-l\left((\omega)_{\mathcal{K}_n}-E_j\right) = \deg(E_j) - g_{\mathcal{K}_n} + 1,\] and this yields 
\begin{equation}\label{eq:lEj}
l(E_j) = g_{\mathcal{K}_n} -1,
\end{equation}
which is independent of $j$. The above holds, since $D$ is of degree $0$ and $\deg(E_j) = 2g_{\mathcal{K}_n}-2 $. 
Furthermore, $l(-jD)=0$ since $jD$ is not principal. 
Indeed, any non-constant function $f$ in $\mathcal{K}_n$ satisfying $(f)_{\mathcal{K}_n}\geq jD$ must satisfy $(f)_{\mathcal{K}_n}=jD$,
because both are divisors of degree $0$. However, from the cyclic extension $\mathcal{R}_n/\mathcal{K}_n$, 
the minimum positive integer $j$ such that $jD = (f)_{\mathcal{K}_n}$ is $j=n$, a contradiction.

We will now utilize the spaces $L(E_j)$ to decompose $H^0(X_n, \Omega_{\overline{\Q}})$ in the following proposition.

\begin{proposition}\label{prop:H0_vector_spaces}
   For notations as above,
  we have the following isomorphism between $\overline{\Q}$-vector spaces.
  \begin{equation}\label{eq:decomp_holomorphic_differentials}
  H^0(X_n, \Omega_{\overline{\Q}}) \cong H^0(Y_n, \Omega_{\overline{\Q}}) \oplus 
  \bigoplus\limits_{j=1}^{n-1}\varepsilon^j L((\omega)_{\mathcal{K}_n}+jD).
  \end{equation}
\end{proposition}
Such decompositions between unramified extensions of function fields have been studied before of course. We refer to the work of Tamagawa \cite[Case I]{Tamagawa:51} 
as a guideline for the following proof.
\begin{proof} Denote by $\pi : X_n \rightarrow Y_n$ the unramified morphism of curves, 
  with Galois group $\Gal(X_n/Y_n)=\Gal(\mathcal{R}_n/\mathcal{K}_n) = \langle c \rangle \cong \Z/n\Z$. Since 
  $\mathcal{R}_n/\mathcal{K}_n$ is a cyclic Kummer extension of degree $n$ generated by $\varepsilon$, the space of $\overline{\Q}$-differentials on the Heisenberg curve 
  $X_n$ is spanned by elements 
  of the form $\varepsilon^j f(x,y)dx$, where $j=0,\ldots,n-1$ and $f \in \mathcal{K}_n$. Furthermore, the cyclic generator $c$ decomposes the space of $\overline{\Q}$-differentials 
  into $n$ eigenspaces, since it maps $c(\varepsilon^j) = \zeta^j\varepsilon^j$ and $c(f)=f$, for any $f\in \mathcal{K}_n$. Thus, any (meromorphic) differential $w$ on the Heisenberg curve $X_n$ can be uniquely 
  written as $w = \sum_{j=0}^{n-1} \varepsilon^j w_j$, where the $w_j$ are (meromorphic) differentials on the Fermat curve $Y_n$. 
  
  We can prove that $w$ is a holomorphic differential on $X_n$ if and only if each $\varepsilon^j w_j$ is holomorphic on $X_n$. The converse direction is obvious. 
  For the first direction, consider $c$ as an automorphism of $X_n$ and let $c^*$ denote its pullback, which acts on the space of holomorphic differentials of $X_n$. We can write
  \[\frac{1}{n} \sum\limits_{i=0}^{n-1}\zeta^{-ij}\left(c^*\right)^i(w) = \frac{1}{n}\sum\limits_{i=0}^{n-1} \sum\limits_{\lambda=0}^{n-1} \zeta^{i(\lambda-j)} \varepsilon^\lambda w_\lambda = \varepsilon^j w_j.\]
  Thus, if $w$ is holomorphic, then every $\left(c^*\right)^i(w)$ is also holomorphic and the above equality implies the same for each $\varepsilon^j w_j$.
  We will now work in terms of divisors on both curves $X_n$ and $Y_n$. Since $\pi$ is unramified, the canonical divisor of $X_n$ is exactly the pullback of the canonical divisor of $Y_n$, see for instance 
  \cite[IV, Proposition 2.3]{Hartshorne:77}. This concludes the ($j=0$) component $H^0(Y_n, \Omega_{\overline{\Q}})$ in the desired 
  decomposition (\ref{eq:decomp_holomorphic_differentials}). Furthermore, the pullback of the divisor $D$ in $\mathcal{K}_n$ is as follows. 
  \[\pi^*(D) = \frac{1}{n}\pi^*( (\varepsilon^n)_{\mathcal{K}_n}) = \frac{1}{n} (n\varepsilon)_{\mathcal{R}_n} = 
  (\varepsilon)_{\mathcal{R}_n}.\] As a consequence, the divisor of $\varepsilon^j w_j$ on $X_n$ is exactly the pullback of the divisor 
  $jD+(w_j)_{\mathcal{K}_n}$ from $Y_n$. This component is holomorphic if and only if 
  \[ (w_j)_{\mathcal{K}_n} \geq -jD,\] and since $w_j$ are differentials on $Y_n$, the above inequality describes precisely the space $L((\omega)_{\mathcal{K}_n}+jD)$ on $\mathcal{K}_n$. The decomposition follows.
\end{proof} Note that from equation (\ref{eq:lEj}), we can verify the dimensions match in 
(\ref{eq:decomp_holomorphic_differentials}). Indeed, it holds that
\[g_n= g_{\mathcal{K}_n} + (n-1)(g_{\mathcal{K}_n}-1) = n(g_{\mathcal{K}_n}-1)+1,\] 
for the genus $g_n$ of the Heisenberg curve $X_n$ as in equation (\ref{eq:genus}), for odd $n$. 

Set 
$V_{n,j}:= \varepsilon^j L( (\omega)_{\mathcal{K}_n}+jD)$ for the twisted Riemann-Roch spaces. We now show 
that this decomposition from Proposition \ref{prop:H0_vector_spaces} is $H_{n}$-equivariant.

\begin{proposition} \label{prop:module_decomp} For notations as above, the $\overline{\Q}$-vector space decomposition 
\[H^0(X_n,\Omega_{\overline{\Q}}) \cong H^0(Y_n,\Omega_{\overline{\Q}}) \oplus 
\bigoplus\limits_{j=1}^{n-1} V_{j,n},\] 
is a decomposition of $\overline{\Q}[H_n]$-modules.
\end{proposition}

\begin{proof}
Recall the definitions of the elements $\alpha,\beta$ and $c$ of $H_n$ as in the 
proof of Proposition \ref{prop:Gal_H_n}. It suffices to prove that 
$\alpha$ preserves each $V_{j,n}$. Fix a $j$ such that $1\leq j \leq n-1$. Then a holomorphic differential 
corresponding to an element in $V_{j,n}$ is of the form
$f(x,y) \varepsilon^j \omega$ such that $(f)_{\mathcal{K}_n}+(\omega)_{\mathcal{K}_n}+ jD \geq 0$ 
in terms of divisors in $\mathcal{K}_n$. The action of $\alpha$ yields 
the element $ \zeta^{j+1} \cdot \varepsilon^j \frac{(1-x)^j}{y^j} f(\zeta x,y) \omega$, thus we need to prove
\[ j(1-x)_{\mathcal{K}_n} - j(y)_{\mathcal{K}_n} + (f(\zeta x,y))_{\mathcal{K}_n} + (\omega)_{\mathcal{K}_n} + j D\geq 0,\] in 
terms of divisors of $\mathcal{K}_n$. Observe that $(f(\zeta x,y))_{\mathcal{K}_n} = \widetilde{\alpha}^*( f)_{\mathcal{K}_n}$, 
as a pullback of the automorphism $\widetilde{\alpha}$ of the Fermat 
curve. Additionally, $\widetilde{\alpha}$ preserves holomorphicity, that is
\begin{equation}\label{eq:divisors1} 
  \widetilde{\alpha}^* (f)_{\mathcal{K}_n} + \widetilde{\alpha}^*(\omega)_{\mathcal{K}_n} + j \cdot  
\widetilde{\alpha}^*D \geq 0,
\end{equation} and $\widetilde{\alpha}^* (\omega)_{\mathcal{K}_n} = (\zeta \omega)_{\mathcal{K}_n} = (\omega)_{\mathcal{K}_n}$. We compute 
\begin{equation}\label{eq:divisors2} \widetilde{\alpha}^* D = 
  \sum\limits_{i=1}^{n-1}i \cdot \beta_{i+1} -
   \frac{n-1}{2} \sum\limits_{i=1}^{n} \gamma_i 
   = D + \left( \frac{1-x}{y} \right)_{\mathcal{K}_n}, \end{equation} since 
   \[  \widetilde{\alpha}^* D - D = (n-1)\cdot \beta_{n} - \sum\limits_{i=1}^{n-1} \beta_i. \] 
   Combining equations (\ref{eq:divisors1}) and (\ref{eq:divisors2}) yields the desired result.   
\end{proof}

\begin{remark}\label{remark:ell_three}
Note that in the case of $n=3$, the Heisenberg curve $X_3$ is of genus $1$ and therefore the decomposition of Proposition
\ref{prop:H0_vector_spaces} is simply $H^0(X_3, \Omega_{\overline{\Q}}) \cong H^0(Y_3, \Omega_{\overline{\Q}})$. That is, the $j$-indexed spaces 
are trivial in this case, a fact which we will frequently reference.
\end{remark}

The decomposition from Propositions \ref{prop:H0_vector_spaces},\ref{prop:module_decomp} will play a crucial role 
in subsequent sections and in the proof of the main theorem. However, for the 
rest of the section, we will provide a combinatorial argument on how to construct a basis of 
holomorphic differentials of $X_n$. We will only use the basis to be able to fully describe all periods in the next section, 
and to have a concrete understanding of the holomorphic differentials. The constructed basis will not be 
required in the proof of our main Theorem 
\ref{thm:main}, and readers focused solely in that may safely skip the remainder of this subsection.

\subsection{Constructing a basis of holomorphic differentials} \label{subsec:basis}

We now provide a combinatorial construction for a basis of holomorphic differentials of $X_n$. A basis for the space 
$H^0(Y_n,\Omega_{\overline{\Q}})$ of the Fermat curve $Y_n$ consists of the differentials $x^\mu y^\lambda \omega$, such that $\mu,\lambda \geq 0$ and 
$\mu+\lambda \leq (2 g_{\mathcal{K}_n} - 2)/n = n-3$, which can be seen by the divisors 
$(x)_{\mathcal{K}_n},(y)_{\mathcal{K}_n}$ and $(\omega)_{\mathcal{K}_n}$ of $\mathcal{K}_n$. 

By Proposition \ref{prop:H0_vector_spaces}, it suffices to complement the Fermat basis 
with a basis for each Riemann-Roch space $L((\omega)+jD)$ for $j=1,\ldots,n-1$, where $\omega = dx/y^{n-1}$ and $D=\frac{1}{n}(\varepsilon^n)_{\mathcal{K}_n}$.

To span the $j=1$ component, we will define a set of differentials $\mathcal{B}^{(1)} = \mathcal{B}^{(1)}_1 \cup \mathcal{B}^{(1)}_2$ consisting 
of pure monomial functions and some specific rational functions respectively. Set $S:=(n-5)/2$, then these are defined as follows.
\[\mathcal{B}^{(1)}_1 := \{ \varepsilon x^\mu y^\lambda \omega \mid \quad \mu,\lambda \geq 0, \ \mu+\lambda \leq S\}.\]
To construct the set of rational functions, we define $f_N:= y^{-N} \prod_{i=0}^{N-1}(x-\zeta^{-i})$ for $1\leq N \leq n-1$, and set: 
\[\mathcal{B}^{(1)}_2 := \{\varepsilon x^\nu f_N \omega \mid \quad 1\leq N \leq n-1, \quad 0 \leq \nu \leq \min\{S,n-N-1\}\}.\]

The elements of $\mathcal{B}^{(1)}_1 $ arise by multiplying the Fermat basis by $\varepsilon$. By considering the new 
bound at infinity from the divisor $D_\infty$, these remain holomorphic if and only if $\mu+\lambda \leq S$. Thus we have 
$|\mathcal{B}^{(1)}_1| = (S+1)(S+2)/2$ holomorphic differentials.

Since these are insufficient to span the entire twisted Riemann-Roch space, the elements of $\mathcal{B}^{(1)}_2$ 
complement this, by introducing differentials with controlled poles. According to the zero divisor $D_0$ of $D$, we are permitted up to $i$ poles at 
$\beta_i=P_{(x=\zeta^{-i},y=0)}$ for $1 \le i \le n-1$, while requiring no poles at 
$\beta_n = P_{(x=1,y=0)}$. In the definition of $f_N$, each factor $(x-\zeta^{-i})$ contributes 
$n$ zeros at $\beta_i$, while $1/y$ contributes a simple pole at every $\beta_i$, keeping us within the 
permitted bounds. The factor $x^\nu$ is now introduced to exhaust the possible bound, up to $S$, for the poles at infinity. 
Thus $\nu \geq 0$ is capped at $\min\{S,n-N-1\}$ to avoid linear dependence in the numerators of the differentials, 
arising from the relation 
$x^n+y^n=1$.

Therefore, to properly count the differentials of $\mathcal{B}^{(1)}_2$, we have to consider the two cases 
$N \in  \{1, \dots, n-S-1\}$ and $N \in \{n-S, \dots, n-1\}$. For all $N$ in the former case, there are 
$S+1$ choices for $\nu$, yielding $(n-S-1)(S+1)$ differentials. For all $N$ in the latter case, the bound for $\nu$ is $n-N-1$, 
thus the number of possible values for $\nu$ decreases as $N$ increases. This yields another $\sum_{i=1}^S i= S(S+1)/2$ 
differentials. 

Combining the sizes of $\mathcal{B}^{(1)}_1$ and $\mathcal{B}^{(1)}_2$, we get exactly 
$|\mathcal{B}^{(1)}| = g_{\mathcal{K}_n} - 1$ differentials. 
Thus $\mathcal{B}^{(1)}$ spans the entire twisted Riemann-Roch subspace $V_{1,n}$.

We adapt this construction now for the $j\geq 2$ components based on Equation (\ref{eq:jD}). 
This permits $ij$ poles at each $\beta_i$, except for $\beta_n$ where no poles are allowed. Twisting by $\varepsilon^j$ 
sets the bound for poles at infinity to be $S_j:= n-3-j(n-1)/2$, which is strictly negative for $j\geq 2$. 
Thus no monomial differentials are allowed in this construction. 
As a consequence, we will build the basis \[ \mathcal{B}^{(j)} := \bigcup\limits_{N=0}^{n-1} \mathcal{B}^{(j)}(N),\] where the sets 
$\mathcal{B}^{(j)}(N)$ will consist entirely of rational functions, indexed by the power
 $N$ of $y$ in the denominator, for $0\leq N \leq n-1$. To compensate for the negative $S_j$ bound, we have to make use 
 of our allowed poles at 
 $\beta_i$ to create zeros at infinity. We do this by allowing factors $(x-\zeta^{-i})$ to be in the denominator. 

Define the pole capacity $c_i$ of $\beta_i$ according to Equation (\ref{eq:jD}) as:
\[c_i = c_i(N) := \max\{\kappa \in \mathbb{Z}_{\ge 0} \mid ji \ge \kappa n + N\} = \left\lfloor \frac{ji - N}{n} 
\right\rfloor, \quad c(N) := \sum_{i=1}^{n-1} c_i(N).
\] Each value $c_i$ dictates the maximum power of $(x-\zeta^{-i})$ permitted in the denominator alongside $y^N$. Set also
\[
d_N := \#\{i \in \{0,\dots,n-1\} : N > ij\}.
\] The value $d_N$ is the 
numerator degree required to cancel the unwanted poles introduced by $1/y^N$. 

Let $D_k, 0 \le k \le c_N$ be the partial product formed by taking exactly the first $k$ linear factors of 
the polynomial $D_{\max} := \prod_{i=1}^{n-1} (x - \zeta^{-i})^{c_i(N)}$ when expanded strictly in order from 
$i=1$ to $n-1$ (with $D_0 := 1$). Define the base rational functions $f_N := y^{-N} \prod_{i=0}^{d_N-1} (x - \zeta^{-i})$ 
and $M_N := d_N - N - S_j$. The basis elements $\mathcal{B}^{(j)}(N)$ are defined as follows:

\begin{align*} \textrm{If } M_N > 0, \textrm{ set } \mathcal{B}^{(j)}(N) 
  &:= \left\{ \frac{\varepsilon^j f_N \omega}{D_k} \ \middle|\ M_N \le k \le c_N \right\}, \\
\textrm{otherwise, } \mathcal{B}^{(j)}(N) 
&:= \left\{ \varepsilon^j x^\nu f_N \omega \ \middle|\ 1 \le \nu \le |M_N| \right\} 
\cup \left\{ \frac{\varepsilon^j f_N \omega}{D_k} \ \middle|\ 0 \le k \le c_N \right\}.
\end{align*}

%For $N=0$, we define $1/f_0$ by taking 
%the product of the first $|S_j|$ available denominator factors (counting multiplicities up to capacity $c_i$) to satisfy 
%the infinity bound. We can then scale the denominator further, by sequentially appending additional permitted denominator factors up to the capacity $c_N$, rather than scaling the numerator. 
%This provides $c_N + S_j + 1$ basis elements.

%For $N \geq 1$, we set $M := d_N - N - S_j$ and define:
%\[
%f_N := \begin{cases} \frac{\prod\limits_{i=0}^{d_N-1}(x-\zeta^{-i})}{y^N \prod\limits_{i=1}^{\deg=M} (x-\zeta^{-i})}, & \text{if } M > 0, \\
%  & \\  \frac{\prod\limits_{i=0}^{d_N-1}(x-\zeta^{-i})}{y^N}, & \text{if } M \le 0. \end{cases}
%\]
%If $M \ge 0$, we utilize $M$ denominator factors to satisfy the infinity bound, leaving $c_N - M$ available scalings. 
%If $M < 0$, the numerator degree already satisfies the infinity bound, allowing us to approach the bound by scaling 
%the numerator by $x^\nu$ (for $\nu \le |M|$), or to scale the denominator utilizing the capacity $c_N$. 
In both cases, $\mathcal{B}^{(j)}(N)$ has exactly $c_N - M_N + 1$ elements. Substituting $M_N$, we have in total:
\[
\sum_{N=0}^{n-1} (c_N - d_N + N + S_j + 1) = \frac{1}{2}n(n-3)= g_{\mathcal{K}_n}-1,
\] elements for each $j$-component, as required. We shall showcase that the sum of $c_N - d_N$ is equal to $n(n-1)(j-2)/2$ from 
which the above equality follows by computation.
Using the Iverson bracket $[P]$ to denote $1$ if a proposition $P$ is true and $0$ otherwise, we evaluate the sum of $d_N$:
\[
\sum_{N=1}^{n-1} d_N = \sum_{N=1}^{n-1} \sum_{i=0}^{n-1} [N > ij] = \sum_{i : ij \leq n-1} (n - 1 - ij).
\]
Furthermore, using Hermite's identity, the sum over $c_N$ becomes:
\[
\sum_{N=0}^{n-1} c_N = \sum_{i: ij \geq n}^{n-1}\sum_{N=0}^{n-1}  \left\lfloor \frac{ij - N}{n} \right\rfloor = \sum_{i: ij \ge n}^{n-1} \sum_{\kappa=0}^{n-1} \left\lfloor \frac{(ij - n + 1) - \kappa}{n} \right\rfloor = \sum_{i: ij \ge 0} (ij - n + 1).
\]
Since $d_0 = 0$, the difference is:
\[
\sum_{N=0}^{n-1} (c_N - d_N) = \sum_{i=0}^{n-1} (ij - n + 1) = (j - 2)\frac{n(n+1)}{2},
\] as desired.

Therefore, if $\mathcal{B}^{(0)}$ denotes the Fermat holomorphic differentials, a basis $\mathcal{B}$ for the holomorphic differentials of the Heisenberg curve $X_n$ is 
\[ \mathcal{B} := \bigcup\limits_{j=0}^{n-1} \mathcal{B}^{(j)}.\]
We will showcase the entire construction in the following two examples.

\noindent \textbf{Example for $\mathbf{n=5}$:} The $6$ holomorphic differentials that are lifted from the Fermat curve are 
\[\omega, x\omega,x^2\omega,y\omega,y^2\omega ,xy\omega,\] and the rest of the basis of $H^0(X_n,\Omega_{\overline{\Q}})$ appears in the following table:
\begin{table}[h!]
\centering
\setlength{\tabcolsep}{2pt}
\renewcommand{\arraystretch}{1.5}
\caption{Basis of holomorphic differentials for $n=5$ across the $j=1$ to $j=4$ components.}
\label{table:1}
\begin{tabular}{|l|l|l|l|}
\hline
$j=1$ & $j=2$ & $j=3$ & $j=4$ \\
\hline
 $ \varepsilon \omega$ & $\frac{1}{(x - \zeta)(x - \zeta^2)} \varepsilon^{2} \omega$ & $\frac{1}{(x-\zeta)^2(x-\zeta^2)(x-\zeta^3)}\varepsilon^3 \omega$ &  $\frac{1}{(x-\zeta)^3(x-\zeta^2)^2(x-\zeta^3)}\varepsilon^4 \omega$\\
 $\frac{(x - 1)}{y} \varepsilon \omega$ & $\frac{(x - 1)}{y(x - \zeta)(x - \zeta^{2})} \varepsilon^{2} \omega$ & $\frac{(x-1)}{y(x-\zeta)^2(x-\zeta^2)(x-\zeta^3)}\varepsilon^3 \omega$ &  $\frac{(x-1)}{y(x-\zeta)^3(x-\zeta^2)^2(x-\zeta^3)}\varepsilon^4 \omega$\\
 $\frac{(x - 1)(x - \zeta^{-1})}{y^{2}} \varepsilon \omega$ & $\frac{(x - 1)}{y^{2}(x - \zeta)} \varepsilon^{2} \omega$ & $\frac{(x-1)}{y^2(x-\zeta)^2(x-\zeta^2)}\varepsilon^3 \omega$ & $\frac{(x-1)}{y^2(x-\zeta)^2(x-\zeta^2)^2(x-\zeta^3)}\varepsilon^4 \omega$ \\
 $\frac{(x - 1)(x - \zeta^{-1})(x - \zeta^{-2})}{y^{3}} \varepsilon \omega$ & $\frac{(x - 1)(x - \zeta^{-1})}{y^{3}(x - \zeta)} \varepsilon^{2} \omega$ & $\frac{(x - 1)}{y^{3}(x - \zeta)(x - \zeta^2)} \varepsilon^{3} \omega$ &  $\frac{(x-1)}{y^3(x-\zeta)^2(x-\zeta^2)(x-\zeta^3)}\varepsilon^4 \omega$\\
 $\frac{(x - 1)(x - \zeta^{-1})(x - \zeta^{-2})(x - \zeta^{-3})}{y^{4}} \varepsilon \omega$ & $\frac{(x - 1)(x - \zeta^{-1})}{y^{4}} \varepsilon^{2} \omega$ & $\frac{(x - 1)(x - \zeta^{-1})}{y^{4}(x - \zeta)(x - \zeta^{2})} \varepsilon^{3} \omega$ & $\frac{(x-1)}{y^4(x-\zeta)^2(x-\zeta^2)}\varepsilon^4 \omega$ \\[2ex]
\hline
\end{tabular}
\end{table}

\noindent \textbf{Example for $\mathbf{n=7}$:}
Let us showcase the $H_n$-action on the holomorphic differentials corresponding to the component $\varepsilon L((\omega)_{\mathcal{K}_n}+D)$, for $n=7$. Denote by $f_N$ the 
rational functions 
\[f_N =  1/y^N \prod_{m=0}^{N-1}(x-\zeta^{-m}), \quad 1\leq N \leq 6.\] The $14$ basis elements of the $j=1$ component can be partitioned into two sets of $7$ by:
\[ A_1 = \{\varepsilon \omega, f_1 \varepsilon \omega, f_2\varepsilon \omega ,\ldots, f_6 \varepsilon \omega\}, \  \textrm{ and }  
\ A_2 = \{x\varepsilon \omega, y\varepsilon \omega, xf_1\varepsilon \omega ,\ldots, xf_5 \varepsilon \omega\}.\]
Each of the two sets spans a $\overline{\Q}$-vector space that is $H_7$-invariant. To see this, recall the definitions of the generators 
$\alpha,\beta$ from the proof of Proposition \ref{prop:Gal_H_n}. It suffices to verify the action of $\alpha$. Let 
$\nu_1 = \varepsilon\omega(1,f_1,\ldots,f_6)$ be the row vector of the elements in $A_1$. The action of $\alpha$ on this vector can be represented as follows.
\[\alpha(\nu_1) = \nu_1 \cdot 
\left(
\begin{array}{c|c}
  0 & 1 \\
  \hline
  -\zeta I_{6\times 6} & 0
\end{array}
\right),\]where $I_{6\times 6}$ is the indentity matrix.

For the second part, the action of $\alpha$ on $xf_5 \varepsilon \omega$ produces $xf_6\varepsilon \omega$. However, this remains in the span 
of $A_2$ as $xf_6 \varepsilon \omega = (y-\zeta f_6)\varepsilon \omega$. Thus, we write $\nu_2 = x\varepsilon\omega(1,f_1,\ldots,f_6)$ and similarly we have that 
\[\alpha(\nu_2) = \nu_2 \cdot 
\left(
\begin{array}{c|c}
  0 & -\zeta \\
  \hline
  D_{6\times 6} & 0
\end{array}
\right),
\] where $D_{6\times 6}= - \mathrm{diag}(\zeta^{2},\zeta^{3},\zeta^{4},\zeta^{5},\zeta^{6},-1)$. Overall, we see that the $j=1$ component is 
$H_7$-invariant, but not irreducible. However, since the irreducible representations of $H_7$ strictly have dimensions $1$ or $7$, recall equation 
(\ref{eq:chi_ijs}), the two subspaces spanned by $A_1$ and $A_2$ respectively are irreducible $\overline{\Q}[H_7]$-modules.

\section{Periods}\label{sec:periods}
In this section, we provide a convenient method 
for calculating the periods of the Heisenberg curves $X_{\ell^n}$. 
For any holomorphic differential $u$ in $H^0(X_{\ell^n}, \Omega_{\mathbb{C}})$ 
determined previously in \S\ref{subsec:basis}, these periods are given by the integrals:
\[ \int\limits_{ \gamma}u, \quad \gamma \in H_1(X_{\ell^n},\Z).  \] 
In order to obtain a proper 
period matrix, one initially requires a basis for 
$H_1(X_{\ell^n},\Z)$, and a change of basis to make it 
symplectic. 
A classical approach to obtain a basis is to use the 
Reidemeister-Schreier method
\cite[Thm. 9.1]{bogoGrp} 
(see also \cite{kontogeorgis2024galoisactionhomologyheisenberg} 
for this specific application). Additionally, alternative homology bases are provided 
in \cite{Banerjee_Merel_2024}. Furthermore, Streit's algorithm in 
\cite{StreitSymplecticRepresentations} can be used directly to computationally 
generate a 
symplectic basis for any Heisenberg curve $X_{\ell^n}$, from which one would construct 
the corresponding period matrix. 

To be able to compute periods, we will follow the method from the Appendix by Rohrlich in the article of Gross 
\cite{GrossAbelianIntegrals} for the case of Fermat curves. However, Rohrlich uses explicit contour integration for the paths in the homology of the Fermat curve, which procedure can 
become very complicated on Heisenberg curves. 
Thus, we will firstly generalize the approach by Rohrlich using Fox calculus on free groups. The Fox derivatives, as these will be defined later, 
will automate the tracking of winding numbers around the punctures $0$ and $1$ and make the contour integration ``an algebraic procedure'', 
by encoding it as an element in an automorphism group ring. For instance, we can recover the Fermat periods with ease in Equation (\ref{eq:fermat_periods}), from Equation (\ref{eq:period}). 
Our procedure will easily generalize to Heisenberg curves and we will be able to write any period arising from any holomorphic differential, integrated over any 
generator of a symplectic homology basis.

First, recall that the topological fundamental 
$\pi_1(\mathbb{C}-\{0,1\},z_0)$ is freely generated by two elements, say 
$l_\alpha$ and $ l_\beta$, which correspond to the 
homotopy classes of small, positively oriented loops 
around the punctures $0$ and $1$. We assume the basepoint $z_0$ 
lies in the interval 
$(0,1)$. Let $X_{\ell^n}^\circ$ denote the affine Heisenberg curve which is a topological 
cover of $\mathbb{C}-\{0,1\}$. For any point $P_0$ above $z_0$, 
the image of $\pi_1(X_{\ell^n}^\circ,P_0)$ in 
$\pi_1(\mathbb{C}-\{0,1\},z_0)$ is the normal closure of the 
subgroup generated by the 
elements $l_\alpha^{\ell^n},l_\beta^{\ell^n}, \ [l_\alpha, [l_\alpha,l_\beta]]$ 
and $[l_\beta, [l_\alpha,l_\beta]]$. 

Passing to the 
homology $H_1(X_{\ell^n},\Z)$ of the projective curve $X_{\ell^n}$, 
the elements $l_\alpha^{\ell^n},l_\beta^{\ell^n}$ become 
trivial. 
Group-theoretically, this occurs 
because uniformizing the projective curve requires 
passing to a subgroup of the triangle group 
$\Delta(\ell^n,\ell^n,\ell^n)$, as in \S\ref{sec:heis_curves}, where these corresponding 
elements appear as relations. Furthermore, $H_1(X_{\ell^n},\Z)$ is 
generated as a $\Z[H_{\ell^n}]$-module by $[l_\alpha, [l_\alpha,l_\beta]]$ 
and $[l_\beta, [l_\alpha,l_\beta]]$. If 
$\ell^n \neq 3$, this set of generators is minimal; however, 
for the $X_3$ curve a single generator suffices. 
See \cite[Prop. 16]{kontogeorgis2024galoisactionhomologyheisenberg} and 
compare this with Theorem \ref{thm:intro}. 
Thus, $H_1(X_{\ell^n},\Z)$ is spanned by elements of the form:
\begin{align}\label{eq:span_homology}
l_\beta^s [l_\alpha,l_\beta]^j l_\alpha^i \cdot [l_\alpha, [l_\alpha,l_\beta]] \cdot (l_\beta^s [l_\alpha,l_\beta]^j l_\alpha^i)^{-1}, & \quad  \textrm{ for } 0\leq i,j,s \leq \ell^n-1 \\
l_\beta^s [l_\alpha,l_\beta]^j l_\alpha^i \cdot [l_\beta, [l_\alpha,l_\beta]] \cdot (l_\beta^s [l_\alpha,l_\beta]^j l_\alpha^i)^{-1}. & \nonumber
 \end{align} 
 These elements are precisely the images of the base paths 
 $\gamma_1:=[l_\alpha, [l_\alpha,l_\beta]]$ and 
 $\gamma_2:=[l_\beta, [l_\alpha,l_\beta]]$ under the action of 
 the elements
 $\beta^s c^j \alpha^i$ in $H_{\ell^n}$ on $H_1(X_{\ell^n},\Z)$, where the action is given by: 
 \[\beta^s c^j \alpha^i \cdot \gamma = l_\beta^s [l_\alpha,l_\beta]^j l_\alpha^i \cdot \gamma \cdot (l_\beta^s [l_\alpha,l_\beta]^j l_\alpha^i)^{-1}.\]
 
To see why this holds, we can rely on standard covering space theory. Let 
 $X\rightarrow Y$ be a normal finite topological cover of path-connected spaces. If $\pi_1$ is 
 the fundamental group of $Y$ and 
$N$ is the image of the fundamental group of $X$ under the covering map, then the quotient $\pi_1/N$ is naturally 
isomorphic to the group of deck transformations $\mathrm{Deck}(X/Y)$. 

The group $\pi_1/N$ acts on $N$ by outer conjugation. When we pass to the abelianization $N/[N,N]$ 
all inner automorphisms become trivial, making this outer action well-defined. Since $N/[N,N]$ is isomorphic to 
$H_1(X,\Z)$, we get a conjugation action on the homology cycles.

We want to show that this conjugation action matches 
the standard action of a deck transformation $f$ in $\mathrm{Deck}(X/Y)$ on a homology cycle. A natural way 
to see this is by working with the universal covering space $\widetilde{Y}$ of $Y$, 
where we can identify $N$ with $\mathrm{Deck}(\widetilde{Y}/X)$. For any lift 
$f^\prime : \widetilde{Y}\rightarrow \widetilde{Y}$ (corresponding to the representative in $\pi_1/N$) we have that 
$f^\prime \mathrm{Deck}(\widetilde{Y}/X)f^{\prime -1} = \mathrm{Deck}(\widetilde{Y}/f(X))$. Although 
the deck group remains the same, as $f(X)=X$, the points of $X$ have been translated by $f$. Consequently the same translation 
happens to the homology cycles.

To apply this in our setting, we pick $X=X_{\ell^n}^\circ$ and $Y=\mathbb{P}^1-\{0,1,\infty\}$. Recall that the homology of the projective curve
$H_1(X_{\ell^n},\Z)$ is a quotient of $H_1(X_{\ell^n}^\circ,\Z)$ by trivializing the 
cycles corresponding to loops around the ramified points above $0,1,\infty$ in the Belyi covering map. Furthermore, 
the action of $H_{\ell^n}$ (as $\mathrm{Deck}(X/Y))$ on $H_1(X_{\ell^n}^\circ,\Z)$, passes 
to the quotient $H_1(X_{\ell^n},\Z)$ (as $\Gal(X_{\ell^n}/\mathbb{P}^1)$) by filling the punctures.

 Having established the possible homology generators in (\ref{eq:span_homology}) as a $\Z$-module, Streit's algorithm 
 \cite{StreitSymplecticRepresentations} can carefully pick 
 $2g_{\ell^n}$ elements from this list and form a symplectic basis. As previously, 
 $g_{\ell^n}$ is the genus as in Equation (\ref{eq:genus}). We can continue 
 with our procedure for any of the above possible elements 
 of any possible symplectic basis as follows. 
 
 For any curve automorphism $f $ in $H_{\ell^n}$ of $X_{\ell^n}$ denote 
 by $f^*$ its pullback to the space of holomorphic differentials. We have that 
 \[ \int\limits_{f(\gamma)} u = \int\limits_{\gamma} f^* u,\] which means we only need to explicitly compute the integrals 
 over $\gamma_1,\gamma_2$. This substitution rule is also used by Rohrlich on the Fermat case in \cite{GrossAbelianIntegrals}; see 
 also the integration on dessins part of \cite{WolfartCMJacobians}.

We deviate from Rohrlich's method now and avoid cutting $\gamma_1$ and $\gamma_2$ into multiple 
line segments in $\mathbb{C}-\{0,1\}$, as it was done in \cite{GrossAbelianIntegrals} for the element 
$[l_\alpha,l_\beta]$ regarding Fermat curves. Instead, 
we make the following observations. When computing a Fermat period over a holomorphic differential $u$, 
looping once in the positive direction around $0$ (resp. $1$) 
corresponds to the action $x\mapsto \zeta \cdot x$ (resp. $y\mapsto \zeta \cdot y)$ on $u$. Thus, in 
order to keep track of the winding numbers around $0$ (resp. $1$) 
we have to keep track of the automorphism sequence applied 
on the differential $u$. For instance, travelling along the commutator 
loop $l_\alpha l_\beta l_\alpha^{-1} l_\beta^{-1}$ yields the sequence
\[ \widetilde{\alpha}^* (u), \ \widetilde{\beta}^* \widetilde{\alpha}^*(u), \ (\widetilde{\alpha}^{-1})^* \widetilde{\beta}^* \widetilde{\alpha}^* (u) =  \widetilde{\beta}^*(u), \ (\widetilde{\beta}^{-1})^*(\widetilde{\alpha}^{-1})^* 
\widetilde{\beta}^* \widetilde{\alpha}^*(u) = u,\] 
where  
$\widetilde{\alpha},\widetilde{\beta}$ are the Fermat curve automorphisms as 
in \S\ref{sec:heis_curves}. Given our affine models of the Heisenberg curves in \S\ref{sec:heis_curves}, 
the same principle holds for their periods as the actions $x\mapsto \zeta \cdot x$ and 
$y\mapsto \zeta \cdot y$ extend naturally to the automorphisms $\alpha$ and 
$\beta$ of $H_{\ell^n}$. 

At this point, we must carefully choose the integration paths based at 
$z_0$. To achieve this efficiently, we will implement the idea of the 
{\em tangential basepoint}; specifically, we pick $z_0$ in 
$(0,1)$ to be infinitesimally close to $1$, viewed as a tangent vector towards $0$. That is we take the limit $z_0\rightarrow 1$, and standard limiting arguments work due to the differentials being holomorphic. Therefore, we can write the periods in terms of integrals evaluated over $(0,1)$.

%{\color{blue} maybe add picture here.}

In the base space $\mathbb{C}-\{0,1\}$, with this setup looping around $0$ once is equivalent to 
travelling from $1$ to $0$ without altering the winding states of any puncture, 
and then returning from $0$ to $1$ by altering the winding state of $0$ using the pullback action of the automorphism 
$\alpha$. Conversely, 
the loop from $z_0$ around $1$ does not contribute anything to the path integral, as there is no distance travelled. However, 
it adjusts the winding number of $1$ by the action of the automorphism 
$\beta$. This procedure naturally leads to the Fox derivatives, which are defined as follows.

Let $F$ be a free group of finite rank on $x_1,\ldots,x_r$ and denote by $\epsilon_{\mathbb{Z}[F]}$ the augmentation map 
$\mathbb{Z}[F] \rightarrow \mathbb{Z}$. The Fox derivatives are the $\Z$-linear maps
 $\partial/\partial x_i: \Z[F]\rightarrow \Z[F]$ satisfying the properties:
\begin{align*}
    (1) \quad & \frac{\partial x_i}{\partial x_j} = 1 \quad \textrm{ if } i=j, \ 0 \textrm{ otherwise.} \\
    (2) \quad & \frac{\partial(y_1 + y_2)}{\partial x_j} = \frac{\partial y_1}{\partial x_j} + \frac{\partial y_2}{\partial x_j}, \quad \frac{\partial(m y_1)}{\partial x_j} = m\frac{\partial y_1}{\partial x_j} \quad (y_1, y_2 \in \mathbb{Z}[F], m \in \mathbb{Z}). \\
    (3) \quad & \frac{\partial(y_1y_2)}{\partial x_j} = \frac{\partial y_1}{\partial x_j}\epsilon_{\mathbb{Z}[F]}(y_2) + y_1\frac{\partial y_2}{\partial x_j} \quad (y_1, y_2 \in \mathbb{Z}[F]). \\
    (4) \quad & \frac{\partial f^{-1}}{\partial x_j} = -f^{-1}\frac{\partial f}{\partial x_j} \quad (f \in F).
\end{align*}
For a detailed exposition, their pro-$\ell$ versions and applications we 
refer to the book by Morishita \cite{Morishita2011-yw}. 
It should be noted that the pro-$\ell$ versions were 
heavily used by Ihara in \cite{Ihara1986towers,Ihara3point}. 
Additionally, some 
results of Ihara are reinstated by Nakamura in \cite{NakamuraTangential} 
based on Fox derivatives in combination with tangential basepoints. The latter were 
introduced by Deligne in \cite{Deligne3point} in a much broader setting and were communicated to Ihara, as he notes in 
\cite{Ihara1985-it}. This story and how the two notions interact in \cite{NakamuraTangential} is what inspired our current approach to 
compute periods. 

To formalize, let $F_2$ be the free group on 
$l_\alpha, l_\beta$ and $\pi:F_2 \rightarrow H_{\ell^n}$ (resp. $\pi: F_2 \rightarrow (\Z/\ell^n\Z)^2$) 
be the projection defined by $l_\alpha,l_\beta \mapsto \alpha,\beta$ (resp. $\widetilde{\alpha},\widetilde{\beta}$). By 
placing the tangential basepoint at $1$, we are differientating by 
the loop $l_\alpha$. This happens since the loop
 $l_\beta$ does not contribute to the integral and simply adjusts 
 the winding state for the next loop, which the Fox derivative keeps track of. This is encoded as follows.

 \[ \frac{\partial \left(l_\beta \cdot \gamma\right) }{\partial l_\alpha} = \frac{\partial l_\beta}{\partial l_\alpha} + l_\beta \cdot \frac{\partial \gamma }{\partial l_\alpha} = l_\beta \cdot \frac{\partial \gamma }{\partial l_\alpha}, \] and 
 \[\frac{\partial \left(l_\alpha \cdot \gamma\right) }{\partial l_\alpha} = 1 + l_\alpha \cdot \frac{\partial \gamma }{\partial l_\alpha},\] for $\gamma$ an element in $F_2$. 
 To get an integral from $0$ to $1$ over a loop $\gamma$, we first extend linearly 
 $\pi$ to $\Z[F_2]$ to get an automorphism action sequence corresponding to 
 the generators appearing in $\gamma$. We then use the pullbacks of the automorphisms 
 on the differentials. We 
 multiply by $\alpha^*-1$ (resp. $\widetilde{\alpha}^*-1$) 
 for the paths from $1$ to $0$, which does not alter the winding state around 
 a puncture, and from $0$ back to $1$ which encircles $0$ once. 
 
 For the Heisenberg curves $X_{\ell^n}$, any holomorphic differential $u$ and 
 any loop $\gamma$ in $F_2 \cong \pi_1(\mathbb{C}-\{0,1\},z_0)$ that is non-zero in the 
 homology $H_1(X_{\ell^n},\Z)$, we have

\begin{equation}\label{eq:period}
  \int_{\gamma} u = \int_{0}^{1} (\alpha^*-1) \cdot \pi 
  \left(\frac{\partial \gamma }{\partial l_\alpha} \right)^* (u).
\end{equation}
The same equation holds for the Fermat curves, by 
replacing the automorphism $\alpha$ by $\widetilde{\alpha}$. 
Let us perform the calculation from the appendix of \cite{GrossAbelianIntegrals} 
on Fermat curves of level $d$, using this 
machinery. We have
\[ \frac{\partial [l_\alpha,l_\beta]}{\partial l_\alpha} = 1 - l_\alpha^{-1} l_\beta l_\alpha, \] which projects 
to $1-\widetilde{\beta}$, for $\widetilde{\alpha},\widetilde{\beta}$ being the maps $(x,y)\mapsto (\zeta\cdot x,y)$ 
and $(x,y)\mapsto (x,\zeta \cdot y)$, respectively in 
$(\Z/d\Z)^2$. Recall that we use 
$x=t^{1/d}, \ y = (1-t)^{1/d}$ as variables over the projective $t$-line. Then,
\begin{equation} \label{eq:fermat_periods}
  \int\limits_{[l_\alpha,l_\beta]} x^{r-1}y^{s-d} dx = \int_{0}^{1} (\widetilde{\alpha}^*-1)(1-\widetilde{\beta}^*) 
x^{r-1}y^{s-d} dx = -(1-\zeta^r)(1-\zeta^s) \frac{B(r/d, s/d)}{d},\end{equation} where $B(m,n)$ denotes the Beta function. 
Let us perform now the calculation on Heisenberg curves $X_{\ell^n}$. We have the Fox derivatives:

\begin{align*}
 \frac{\partial [l_\alpha, [l_\alpha,l_\beta]]}{\partial l_\alpha} & = 1 + l_\alpha - l_{\alpha}^2 l_\beta l_\alpha^{-1} - 
l_\alpha [l_\alpha,l_\beta] l_\alpha^{-1} + 
l_\alpha [l_\alpha,l_\beta] l_\alpha^{-1} l_\beta - [l_\alpha, [l_\alpha,l_\beta]], \\
 \frac{\partial [l_\beta, [l_\alpha,l_\beta]]}{\partial l_\alpha} & = l_\beta - l_\beta l_\alpha l_\beta l_\alpha^{-1} + l_\beta [l_\alpha,l_\beta] 
 - l_\beta [l_\alpha,l_\beta] l_\alpha l_\beta^{-1} l_\alpha^{-1},
\end{align*} which are projected in $\Z[H_{\ell^n}]$ to
\[
\alpha - \alpha c \beta - c + c\beta, \quad \beta - \beta^2 c + \beta c - 1, \quad \textrm{ for } 
c=[\alpha,\beta] \in Z(H_{\ell^n}).\]
Therefore, since the pullback action is contravariant, for any holomorphic differential $u$ in $H^0(X_{\ell^n},\Omega_{\mathbb{C}})$ the equation 
(\ref{eq:period}) yields
\[ \int\limits_{[l_\alpha, [l_\alpha,l_\beta]]} u = \int_{0}^{1} (\alpha^*-1) (\alpha^*- \beta^* c^* \alpha^* - c^* + \beta^*c^*)(u), \]
\[\int\limits_{[l_\beta, [l_\alpha,l_\beta]]} u = \int_{0}^{1} (\alpha^*-1) 
(\beta^*- c^* (\beta^2)^* + c^*\beta^* - 1)(u), \] from which all the periods of the Heisenberg curves can be computed. 

\begin{example}

Let us work the period of the holomorphic differential 
$u = \varepsilon dx/y^{\ell^n-1}$ over the path $\gamma_2$. The differential $u$ is an eigenvector of $\beta^*,c^*$ with 
both eigenvalues being $\zeta$ and also $\alpha^*(u)= \zeta \frac{(1-x)}{y} u$. Therefore,
\begin{align*}
\int\limits_{[l_\beta, [l_\alpha,l_\beta]]}\varepsilon \frac{dx}{y^{\ell^n-1}}  
&= -(1-\zeta)(1-\zeta^2) \int_{0}^{1} \left( \zeta \frac{(1-x)}{y} - 1 \right) \varepsilon \frac{dx}{y^{\ell^n-1}} \\
&=  - \zeta (1-\zeta)(1-\zeta^2) \int_{0}^{1} \prod\limits_{i=1}^{\ell^n-1} (1-\zeta^i x)^{(i/\ell^n - 1)} dx  \\ 
& + (1-\zeta)(1-\zeta^2) \int_{0}^{1} (1-x)^{(1/\ell^n-1)}\prod\limits_{i=1}^{\ell^n-1} (1-\zeta^i x)^{( \frac{i+1}{\ell^n} - 1 )} dx.
%&= (1-\zeta)(1-\zeta^2)\left( -\zeta \cdot F_D^{(\ell^n-1)} \left( 1, \left\{ (\ell^n-i)/\ell^n \right\}_{i=1}^{\ell^n-1}, 2 ; 
%\left\{\zeta^i \right\}_{i=1}^{\ell^n-1}\right) \right)
\end{align*} Both of these integrals 
can be written in terms of the integral expansions of the Lauricella hypergeometric series 
$F_D^{(\ell^n-1)}$, see \cite[Example 3.5]{AomotoHypergeometric}.
\end{example}

\begin{remark}\label{remark:periods}
For the Heisenberg curves of level $\ell^n \neq 3$ (with genus 
$g=g_{\ell^n}$ given by Equation (\ref{eq:genus})), Theorems 
\ref{thm:main} and \ref{thm:complement}--which we prove 
later independently of this section--imply the existence of non-zero periods $\int_{\gamma_i} u_i$ and $\int_{\gamma_j} u_j$ whose quotient 
is transcendental. Equivalently, any period quotient of these curves is not defined over $\overline{\Q}$; this quotient is 
well-defined up to the action of the Siegel modular group 
$\mathrm{Sp}(2g, \Z)$ regarding the choice of a symplectic basis of 
$H_1(X_{\ell^n},\Z)$. This means no period matrix representative is an element of the matrix ring $M_g(\overline{\Q})$. For the implication see 
\cite[Thm. 3]{WolfartCMJacobians}, based on \cite[Cor. 2]{WolfartShiga}.
\end{remark}
% - (1-x)^{(1/\ell^n-1)}\prod\limits_{i=1}^{\ell^n-1} (1-\zeta^i x)^{( \frac{i+1}{\ell^n} - 1 )} 

\section{(Co)Homology and Representation Theory}\label{sec:rep_theory}
In this section, we shall discuss the discrete Heisenberg group $H_{\ell^n}$'s action on the homology vector 
spaces $H_1(X_{\ell^n},\overline{\Q})$ and $H_1(X_{\ell^n},\mathbb{C})$ in terms of representation theory, 
where the group acts via automorphisms of the curve as $H_{\ell^n}\cong \Gal(X_{\ell^n}/\mathbb{P}^1)$. As previously,
$\ell$ is an odd prime and $X_{\ell^n}$ is the $\ell^n$-level Heisenberg curve. 
We refer to \cite{Farkas-Kra} for the standard theory of compact Riemann surfaces and their homology groups, as well as their relation with the curves automorphisms. 
In particular, recall that the first homology group $H_1(X_{\ell^n},\Z)$ is a free $\Z$-module of rank 
$2g_{\ell^n}$, for the genus $g_{\ell^n}$ 
as in (\ref{eq:genus}) and $H_1(X_{\ell^n},\overline{\Q}) = H_1(X_{\ell^n},\Z) \otimes \overline{\Q}$ 
as well as $H_1(X_{\ell^n},\mathbb{C}) = H_1(X_{\ell^n},\Z) \otimes \mathbb{C}$ are its extensions of scalars.

Furthermore, we shall discuss the action of $H_{\ell^n}$ on the vector space of the holomorphic 
differentials $H^0(X_{\ell^n},\Omega_{\overline{\Q}})$ and $H^0(X_{\ell^n},\Omega_{\mathbb{C}})$. As our curves are 
defined over $\overline{\Q}$, the $H_{\ell^n}$ action has the same decomposition in both cases with the exception of scalars. 
However, over $\mathbb{C}$ there is a well-known relation between holomorphic differentials of the 
first kind and (co)homology, which builds up from Serre duality, the Hodge principle and the De Rham isomorphism. 
To put everything together, for a compact Riemann surface $\mathcal{C}$, Serre duality asserts that
\[H^1(\mathcal{C}, \mathcal{O}_{\mathcal{C}}) \cong H^0(\mathcal{C}, \Omega_\mathcal{C})^*,\]
where $\mathcal{O}_\mathcal{C}$ is the structure sheaf of $\mathcal{C}$ and $H^0(\mathcal{C},\Omega_\mathcal{C})^*$ is 
the dual of the space of holomorphic differentials $H^0(\mathcal{C},\Omega_\mathcal{C})$, 
which is also denoted sometimes as $\Omega^1(\mathcal{C})$ in the literature. Now, $H^1(\mathcal{C},\mathbb{C})$ (in the singular sense) as a $\mathbb{C}$-vector space is isomorphic to the de Rham first cohomology space
 $H^1_{\mathrm{dR}}(\mathcal{C},\mathbb{C})$ and we have the Hodge decomposition:
\begin{equation}\label{eq:hodge_decomp}
H^1(\mathcal{C},\mathbb{C}) \cong H^1_{\mathrm{dR}}(\mathcal{C},\mathbb{C}) \cong  H^0(\mathcal{C},\Omega_\mathcal{C}) 
\oplus H^1(\mathcal{C},\mathcal{O}_\mathcal{C}),
\end{equation} which splitting is respected by the action of any automorphism of 
$\mathcal{C}$. Additionally, $H^1(\mathcal{C},\mathbb{C})$ is the dual of $H_1(\mathcal{C},\mathbb{C}) = 
H_1(\mathcal{C},\Z)\otimes_\Z \mathbb{C}$ and by the above decomposition splits into holomorphic and anti-holomorphic 
differentials, where the latter are obtained by complex conjugation of the former. The above discussion is, of course, well-known 
and a detailed exposition can be found, for example, in \cite[Ch. 0]{GriffithsHarrisPrincliples}.

Furthermore, $H^1(\mathcal{C},\mathbb{C})$ and 
$H_1(\mathcal{C},\mathbb{C})$ have the same character under 
the action of any $G\subseteq \mathrm{Aut}(\mathcal{C})$. 
By duality, these must have the same character up to complex 
conjugation, and since the $\mathrm{Aut}(\mathcal{C})$-action 
is already well-defined in $H_1(\mathcal{C},\Z)$, the
trace of any matrix representing the action of 
an automorphism
on the homology must be an integer. Thus, the characters coincide.

Based on the decomposition (\ref{eq:hodge_decomp}), we can thus relate the characters of the action of $H_{\ell^n}$ on $H_1(X_{\ell^n},\mathbb{C})$ and on 
$H^0(X_{\ell^n},\Omega_{\mathbb{C}})$, once we get an understanding of them. The character of the former action has been established 
in previous work by the author \cite{kontogeorgis2024galoisactionhomologyheisenberg}. Namely, if we denote $\chi_{\ell^n}$ the character of 
$H_1(X_{\ell^n},\mathbb{C})$ (or over $\overline{\Q}$), we have in terms of irreducible characters (\ref{eq:chi_ijs})
\begin{equation}\label{eq:chi_elln}
\chi_{\ell^n} = \sum\limits_{j=0}^{\ell^n-1}\sum\limits_{i,s=0}^{d_j-1} h_{ijs}\chi_{ijs},
\end{equation}for $d_j = \gcd(\ell^n,j)$ 
and $h_{ijs} = \frac{\ell^n}{d_j} - [i \text{ or } s = 0] - [i=s = 0] - [i+s \equiv 0\mod d_j]$, if $(i,j,s)\neq (0,0,0)$ and $h_{0,0,0}=0$, where we use again the 
Iverson bracket $[P]=1$ or $0$ if the proposition $P$ is true or false. A particularly useful case will be the character $\chi_\ell$ 
at the $\ell$-level Heisenberg curve, that is
\begin{equation}\label{eq:chi_ell}
\chi_{\ell} = \sum\limits_{\substack{i,s=0 \\ i+s  \not\equiv0 \ \mathrm{mod }\ell}}^{\ell-1} \!\!\!\!\! \chi_{i,0,s} +  
\sum\limits_{j=1}^{\ell-1} (\ell-3)\chi_{0,j,0}.
\end{equation}
It will be of importance that the higher-dimensional representations $\chi_{0,j,0}$, for $1\leq j < \ell$ of $H_\ell$ appear with multiplicity $(\ell-3)$. Additionally, 
the first component of one-dimensional representations is credited to the Fermat curve as a subcover of 
$\mathbb{P}^1$, e.g. see \cite[Prop. 2.12]{KontogarParamFermat}. 

We can thus move on understanding the action of $H_{\ell^n}$ on $H^0(X_{\ell^n},\Omega_{\mathbb{C}})$. This will be done as 
an application of the Chevalley-Weil theorem \cite{Chevalley1934-eb}, which is quite powerful in understanding the group-module 
structure of holomorphic differentials, not only of the first kind, in both positive and $0$ characteristic. 
For a convenient description of the theorem, we refer to \cite[Table 1, Thm. 2]{computational_chev_weil}. As it is customary, 
for any finite group $G$, we denote by $\langle \cdot, \cdot \rangle_G$ the standard inner product of characters of $G$, i.e. $\langle \psi_1,\psi_2 \rangle_G = \frac{1}{|G|}\sum_{g\in G} 
\psi_1(g)\overline{\psi_2(g)}$, and for a subgroup $G^\prime$ and a character $\psi$ of $G$, we denote $\mathrm{Res}_{G^\prime}^G\psi$, for the restriction of the representation $\psi$ to $G^\prime$.

Let $\psi_{\ell^n}$ be the character of the action on the holomorphic differentials of the first kind. Then, by (\ref{eq:hodge_decomp}), we have that
\[\chi_{\ell^n} = \psi_{\ell^n} + \overline{\psi}_{\ell^n}.\] 

In \cite{computational_chev_weil} notation, we set $m$ equal to $1$ as we work with 
first kind differentials, and the base of the cover is $\mathbb{P}^1$, which is of genus $0$, with branch points 
$(0),(1),(\infty)$ and ramification indices $\ell^n$. Now, the Chevalley-Weil theorem determines the multiplicity of 
an irreducible character $\chi_{ijs}$ of $H_{\ell^n}$ in $\psi_{\ell^n}$, as follows. 
\begin{equation}\label{eq:chev_weil}
   \langle \psi_{\ell^n}, \chi_{ijs}\rangle_{H_{\ell^n}} = [i=j=s=0] -\dim \chi_{ijs} + 
\sum\limits_{d=0}^{\ell^n-1}\left\langle \frac{-d}{\ell^n}\right\rangle \left(n_{d,(0),\chi_{ijs}} + 
n_{d,(1),\chi_{ijs}} + n_{d,(\infty),\chi_{ijs}}\right),\end{equation}
where $\langle q \rangle= q-\lfloor q \rfloor$ denotes the fractional part of $q$ and the values
 $n_{d,Q,\chi_{ijs}}$, for $Q=(0),(1),(\infty)$ in 
$\mathbb{P}^1$ are defined below. Recall that $\beta,\alpha,(\beta\alpha)^{-1}$, as in the proof of Proposition \ref{prop:Gal_H_n}, 
are the stabilizers of the ramified points lying above $(0),(1),(\infty)$ respectively. Then, we denote by $\psi_Q^d$ the irreducible characters of the 
cyclic groups $\langle \beta\rangle, \langle \alpha \rangle$ and $\langle \alpha\beta \rangle$ accordingly, 
mapping the corresponding cyclic generator to $\zeta^d$, for $\zeta$ being the fixed primitive $\ell^n$-nth root of unity 
and $d=0,1,\ldots,\ell^n-1$. With the above, we have
\[n_{d,(0),\chi_{ijs}}= \langle\psi^d_{(0)}, \mathrm{Res}^{H_{\ell^n}}_{\langle \beta \rangle}(\chi_{ijs}) 
\rangle_{\langle \beta \rangle }, \quad n_{d,(1),\chi_{ijs}}= \langle\psi^d_{(1)}, \mathrm{Res}^{H_{\ell^n}}_{\langle \alpha \rangle}(\chi_{ijs}) 
\rangle_{\langle \alpha \rangle }, \]
\[n_{d,(\infty),\chi_{ijs}}= \langle\psi^d_{(\infty)}, \mathrm{Res}^{H_{\ell^n}}_{\langle \alpha\beta \rangle}(\chi_{ijs}) 
\rangle_{\langle (\beta\alpha)^{-1} \rangle }. \] Using the formula (\ref{eq:chi_ijs}) for the characters $\chi_{ijs}$, we can get the following proposition.
\begin{proposition} \label{prop:n_values}
\[ n_{d,(0),\chi_{ijs}} = \begin{cases}
1, & \textrm{ if } d\equiv s\mod \gcd(\ell^n,j), \\
0, & \textrm{ otherwise,}
\end{cases} \quad n_{d,(1),\chi_{ijs}} = \begin{cases}
1, & \textrm{ if } d\equiv i\mod \gcd(\ell^n,j), \\
0, & \textrm{ otherwise,}
\end{cases}\]
\[n_{d,(\infty),\chi_{ijs}}= \begin{cases}
1, & \textrm{ if } d\equiv -(i+s)\mod \gcd(\ell^n,j), \\
0, & \textrm{ otherwise.}
\end{cases}\]
\end{proposition} \begin{proof} Let us showcase the third computation, as the other two are similar. Set 
$\rho_{j} = \zeta^{\frac{\ell^n}{d_j}}$ a primitive $d_j$-root of unity. 
Observe that for $\lambda=1,\ldots,\ell^n-1$ we have 
$(\beta \alpha)^\lambda = \beta^\lambda c^{\binom{\lambda}{2}} \alpha^\lambda$ and that $\ell^n$ 
divides $j\binom{\frac{\ell^n}{d_j}\lambda}{2}$, since $\ell$ is odd. Recall from equation 
(\ref{eq:chi_ijs}), that only $\ell^n/d_j$-multiples survive as exponents of $\beta$ and $\alpha$ in the 
character $\chi_{ijs}$. Note also that from the monodromy configuration of $(0),(1),(\infty)$ we have that $\psi^d_{(\infty)}( (\beta\alpha)^{-1} ) = \zeta^d$. Therefore,
\begin{align*}
n_{d,(\infty),\chi_{ijs}} & = \frac{1}{\ell^n} \sum\limits_{\lambda=0}^{\ell^n-1} \psi^d_{(\infty)} 
( (\beta\alpha)^{-\lambda}) \overline{\chi_{ijs}( (\beta\alpha)^{-\lambda})} \\
&=\frac{1}{d_j} \sum\limits_{\lambda=0}^{d_j-1} \zeta^{ \frac{\ell^n}{d_j} \lambda (d +i+s) + 
j\binom{ \lambda \frac{\ell^n}{d_j}}{2} } \\
&=\frac{1}{d_j} \sum\limits_{\lambda=0}^{d_j-1} \rho_j^{\lambda (d+i+s)},
\end{align*} and as this is a standard sum of roots of unity, the result follows. \end{proof}

\begin{proposition}[Chevalley-Weil formula for $H_{\ell^n}$-action on $H^0(X_{\ell^n},\Omega_{\mathbb{C}})$]
\label{prop:chev_weil2}
  Let $\chi_{ijs}$ be an irreducible 
character of $H_{\ell^n}$. Let $\psi_{\ell^n}$ be the 
character of $H^0(X_{\ell^n},\Omega_{\mathbb{C}})$ as a 
$\mathbb{C}[H_{\ell^n}]$-module. Then, the character $\chi_{ijs}$ appears in the decomposition of $\psi_{\ell^n}$ with the following multiplicity:
\begin{align}\label{eq:chev_weil2}
\langle \psi_{\ell^n}, \chi_{ijs}\rangle_{H_{\ell^n}} &=
  \frac{\ell^n+3 \gcd(\ell^n,j) - 2( i + s + \langle -(i+s)\rangle_{\gcd(\ell^n,j)})}{2\gcd(\ell^n,j)} \\
  & +[\chi_{ijs}=\chi_{0,0,0}]  - \sum\limits_{\mu = i,s,i+s} [\mu \equiv 0 \mod \gcd(\ell^n,j)], \notag
\end{align} where $\langle \lambda \rangle_{m}$ denotes the unique integer $0\leq \lambda^\prime <m$ such that $\lambda \equiv \lambda^\prime\mod m$.
\end{proposition}
\begin{proof} 
  We will reinterpret the values $n_{d,Q,\chi_{ijs}}$ from Proposition \ref{prop:n_values} as 
  Iverson symbols and plug them into equation 
  (\ref{eq:chev_weil}). Set again $d_j=\gcd(\ell^n,j)$. Observe that 
  $\langle -d/\ell^n\rangle = 1 - d/\ell^n$ for 
  $d=1,\ldots,\ell^n-1$, and for an integer 
  $0\leq \mu < d_j$ we can use the following summation tricks, 
  \[\frac{1}{\ell^n}\sum\limits_{d=0}^{\ell^n-1} d\cdot[d\equiv \mu \! \mod d_j] = 
\frac{1}{\ell^n}\sum\limits_{m = 0}^{\frac{\ell^n}{d_j}-1} (\mu + m \cdot d_j)
=  \frac{2\mu + \ell^n-d_j}{2d_j},\] as well as that 
\[ \sum\limits_{d=0}^{\ell^n-1}[d\equiv \mu \! \mod d_j] = \frac{\ell^n}{d_j}.\]
For a non-trivial character $\chi_{ijs}$, we reinterpret now the right-hand-side of (\ref{eq:chev_weil}) as
\[ -\frac{\ell^n}{d_j} + \sum\limits_{d=0}^{\ell^n-1}\left(1-\frac{d}{\ell^n}\right)
 (n_{d,(0),\chi_{ijs}} + n_{d,(1),\chi_{ijs}} + n_{d,(\infty),\chi_{ijs}}) - 
 \sum\limits_{Q=(0),(1),(\infty)}n_{0,Q,\chi_{ijs}},\]
 which can be evaluated by using the two summation tricks for $\mu=i,s$ and the appropriate representative of $-(i+s)$. 
 The result follows by computation.
\end{proof}

\begin{remark}
Observe that for the Heisenberg curve 
$X_{\ell}$ of level $\ell$, any higher dimensional 
irreducible character $\chi_{0,j,0}$, for $1\leq j <\ell$, appears in the decomposition of the character of 
$H^0(X_{\ell},\Omega_{\mathbb{C}})$ with multiplicity 
$(\ell-3)/2$, by equation (\ref{eq:chev_weil2}). Combined with equation (\ref{eq:chi_ell}), 
this implies the characters $\chi_{0,j,0}$ are split 
evenly between the spaces of holomorphic and anti-holomorphic 
differentials. Furthermore, we see now in the $\ell=7$ 
example, that the $j=1$ component correctly consists of $(\ell-3)/2 =2$ 
irreducibles.  
\end{remark}

With the Chevalley-Weil formula at our disposal and the Hodge decomposition, 
we can in fact derive the character $\chi_{\ell^n}$ in 
(\ref{eq:chi_elln}) from \cite{kontogeorgis2024galoisactionhomologyheisenberg}, independently. It 
is worth noting that in \cite{kontogeorgis2024galoisactionhomologyheisenberg}, 
this was computed via 
combinatorial group theory tools, similar to our derivation of periods. 
These methods were inspired by the techniques used in arithmetic topology \cite{Morishita2011-yw}, 
which is an area about analogies between knots and primes. 
%In particular, 
%the Fox free differential calculus on free groups and its pro-$\ell$ 
%version were heavily used by Ihara in \cite{Ihara1986towers}.

The fact that we can derive the character on the first homology $\mathbb{C}$-vector space, from the character on the holomorphic differentials appears more generally 
in the work of Streit as \cite[Prop. 6]{StreitPeriods}. However, the 
argument is scattered between \cite{StreitHomology,StreitSymplecticRepresentations}. Thus, we will simply adapt his idea to our setting.

By conjugating the one-dimensional characters, we have that
\[n_{d,Q,\overline{\chi}_{ijs}} = \langle\psi^d_{Q}, \mathrm{Res}(\overline{\chi}_{ijs}) \rangle = \langle\overline{\psi^d_{Q}}, \mathrm{Res}(\chi_{ijs}) \rangle  = n_{\ell^n-d,Q,\chi_{ijs}},\]
and based on that and (\ref{eq:chev_weil}), for $\chi_{ijs}\neq \chi_{0,0,0}$, we can compute

\begin{align*}
\langle \chi_{\ell^n}, \chi_{ijs}\rangle &= \langle \psi_{\ell^n} + \overline{\psi}_{\ell^n}, \chi_{ijs}\rangle = 
\langle\psi_{\ell^n}, \chi_{ijs} \rangle + \langle \overline{\psi}_{\ell^n}, \chi_{ijs}\rangle \\
&= \langle\psi_{\ell^n}, \chi_{ijs} \rangle + \langle \psi_{\ell^n}, \overline{\chi}_{ijs}\rangle \\
&= -2\dim\chi_{ijs} + \frac{1}{\ell^n}\sum\limits_{Q=(0),(1),(\infty)}\sum\limits_{d=0}^{\ell^n-1} d\cdot (n_{d,Q,\chi_{ijs}} + n_{\ell^n-d,Q,\chi_{ijs}}) \\
&= -2\dim\chi_{ijs} + \sum\limits_{Q=(0),(1),(\infty)} \left(\sum\limits_{d=0}^{\ell^n-1} n_{d,Q,\chi_{ijs}} - n_{0,Q,\chi_{ijs}}\right) \\ 
&= \dim\chi_{ijs} - n_{0,(0),\chi_{ijs}} - n_{0,(1),\chi_{ijs}} - n_{0,(\infty),\chi_{ijs}}, \\
&
\end{align*}
which coincides with the description of $h_{ijs}$ as intended. In the above, we used the fact that \[\sum\limits_{d=0}^{\ell^n-1} n_{d,Q,\chi_{ijs}} = \dim \chi_{ijs}.\]

\section{Decomposition of the Jacobian}\label{sec:decomp}
In this section, we shall discuss the decomposition of the Jacobian varieties 
$\Jac(X_{\ell^n})$ of the Heisenberg curves $X_{\ell^n}$, with main focus to the base case of $X_\ell$. As always, we identify $X_{\ell^n}$ with the 
complex points of the nonsingular projective algebraic curve given by the affine maps $x^{\ell^n}+y^{\ell^n}=1$ and (\ref{eq:varepsilon_ln}), that means both $X_{\ell^n}$ and $\Jac(X_{\ell^n})$ are defined over $\overline{\Q}$.
For the standard theory of Jacobians of curves, 
we refer to \cite{milneAV} as well as Mumford's books \cite{MumfordAbelian, MR0419430}. Analytically, in terms of the 
complex topology, these are defined as 

\begin{equation} \label{eq:jac_def}
  \Jac(\mathcal{C}) = H^0(\mathcal{C},\Omega_{\mathbb{C}})^*/H_1(\mathcal{C},\Z),
\end{equation}
for any algebraic curve $\mathcal{C}$, see \cite[Thm 2.5, pg. 93]{milneAV} for the precise construction. 
Recall that we have a surjective morphism of curves $\pi:X_{\ell^n} \rightarrow Y_{\ell^n}$, 
where $Y_{\ell^n}$ is the $\ell^n$-level Fermat curve. By pulling back divisors, we have an isogeny:
\begin{equation}\label{eq:pullback}
   \Jac(Y_{\ell^n})\longrightarrow \pi^*(\Jac(Y_{\ell^n})) \subset \Jac(X_{\ell^n}),
\end{equation} where $\pi^*(\Jac(Y_{\ell^n}))$ denotes the pullback of the map $\pi$. Now the Poincaré complete 
reducibility theorem \cite[Cor. 1, pg. 73]{MumfordAbelian} asserts 
that any abelian variety $\mathcal{A}$ is isogenous to a product $\prod_{i=1}^r \mathcal{B}_i^{k_i}$ where $\mathcal{B}_i$ are simple abelian varieties and not isogenous to each other. Based on this and the isogeny (\ref{eq:pullback}), we have 
\begin{equation} \label{eq:jac_decomp}
  \Jac(X_{\ell^n}) \sim \Jac(Y_{\ell^n}) \times \mathcal{A}_1^{k_1} \times \mathcal{A}_2^{k_2} \times \cdots \times \times \mathcal{A}_r^{k_r},
\end{equation}
   where $\sim$ denotes an isogeny, it is an equivalence relation, and the $\mathcal{A}_i$ are simple abelian varieties not isogenous to each other. 
   We shall restrict ourselves to determining properties of the $\mathcal{A}_i$'s as well as to determine the number $r$, since the 
   Jacobians of Fermat curves are already a well-studied topic. See e.g. \cite[Sec. 8.5]{MurtyAbVar}. % Section 8.5

We will rely on the theory developed by Wolfart \cite{WolfartCMJacobians}, based on \cite{WolfartShiga}, 
to link the representation-theoretic approach we have developed so far in this paper with the above decomposition 
(\ref{eq:jac_decomp}). We start with the following lemma of his. 
\begin{lemma}[Wolfart's Lemma 14 \cite{WolfartCMJacobians}] \label{lemma:wolfart_14}
Let $\mathcal{C}$ be a nonsingular projective algebraic curve defined over $\overline{\Q}$, uniformized by 
some cocompact Fuchsian triangle group $\Delta$, 
such that $\Delta/\Gamma = \Gal(\mathcal{C}/\mathbb{P}^1)\subseteq \Aut(\mathcal{C})$, for some finite index subgroup $\Gamma$. Denote by $\Phi$ the canonical representation of $\Delta/\Gamma$ on the $\overline{\Q}$-vector space of holomorphic differentials on $\mathcal{C}$
 (or equivalently on its Jacobian, which is also defined over $\overline{\Q}$). Assume that
 \[\Jac(\mathcal{C}) \sim \mathcal{B}_1^{k_1} \times \cdots \times \mathcal{B}_m^{k_m},\] where the $\mathcal{B}_i, i=1,\ldots,m$ are simple, 
 pairwise non-isogenous abelian varieties. Let $U_{\nu}$ denote the pullback of $H^0(\mathcal{B}_\nu^{k_\nu},\Omega_{\overline{\Q}})$ in $H^0(\Jac(\mathcal{C}),\Omega_{\overline{\Q}})$. 
 Then, every $U_\nu$ is an invariant subspace of the representation $\Phi$.
\end{lemma} 
\begin{proof}
The proof is as described by Wolfart \cite[pg. 25]{WolfartCMJacobians}. Without loss of generality, assume that each $\mathcal{B}_i$ is a subvariety of $\Jac(\mathcal{C})$. The action of $\Delta/\Gamma$ on $\mathcal{C}$ induces an action on $\Jac(\mathcal{C})$, and since the $\mathcal{B}_i$'s are simple and pairwise 
non-isogenous, any automorphism of the curve preserves their isomorphism class.
Consider the possible projections $g_\nu : \Jac(\mathcal{C})\rightarrow \mathcal{B}_\nu$ and denote by 
\[I_\nu := \left( \bigcap\limits_{i\neq \nu} \bigcap\limits_{g_i \in \mathrm{Hom}(\Jac(\mathcal{C}), \mathcal{B}_i)} \ker g_i \right)^\circ,\] 
the connected component of the identity in the intersection of the kernels of all the projections $g_i, i\neq \nu$. Since each $\mathcal{B}_i$ is preserved-up to isogeny-by the action of $\Delta/\Gamma$, 
each subvariety $I_\nu$ is $\Delta/\Gamma$-invariant. As $\Jac(\mathcal{C})$ is defined over $\overline{\Q}$, all the subvarieties, 
the projections and the isogenies are also defined over $\overline{\Q}$, thus $U_\nu \cong H^0(I_\nu,\Omega_{\overline{\Q}})$.
\end{proof}

In the proof of the Lemma \ref{lemma:wolfart_14}, each $I_\nu$ is isogenous to $\mathcal{B}_\nu^{k_\nu}$ and thus is an isotypic component of the action 
of $\Delta/\Gamma$. Therefore, Lemma \ref{lemma:wolfart_14} tells us that in order to understand the decomposition (\ref{eq:jac_decomp}), we have to undestand 
the isotypic components of $H^0(X_{\ell^n},\Omega_{\overline{\Q}})$ as a $\overline{\Q}[H_{\ell^n}]$-module.

In Proposition \ref{prop:module_decomp} we showed that each twisted Riemann-Roch subspace 
$V_{j,\ell^n}$ is $H_{\ell^n}$-invariant. We thus have to determine their isotypic components. 
For now, denote also by $V_{0,\ell^n}$ the holomorphic differentials of 
the Fermat curve. Then, the generator $c$ of the center 
$Z(H_{\ell^n}) = \langle c \rangle$ acts as multiplication by 
$\zeta^j$ on $V_{j,\ell^n}$ for $j=0,1,\ldots,\ell^n-1$. 
As an application of Schur's lemma, e.g. \cite[\S 2.2]{SerreLinear}, 
the action of a central element $g$ on an irreducible representation 
is multiplication by a scalar $\lambda_g$, and different scalars 
$\lambda_g \neq \lambda_g^\prime$ correspond to non-isomorphic 
irreducible representations. Thus, $V_{j,\ell^n},V_{j^\prime,\ell^n}$, 
for $j\neq j^\prime$, have no common irreducible factors, 
as the former is acted upon by scalar multiplication by $\zeta^j$ and the 
latter by 
$\zeta^{j^\prime}$, for $1\leq j,j^\prime \leq \ell^n-1$.

Furthermore, the $j$-indices of $V_{j,\ell^n}$ 
are ``in accordance'' with the 
characters $\chi_{ijs}$, since 
\[\chi_{ijs}(c) = \dim \chi_{ijs} \cdot  \zeta^j, \] which indicates 
the scalar action by $\zeta^j$. Thus, any irreducible character
 $\chi_{ijs}$ from the action on the holomorphic differentials on $X_{\ell^n}$ is strictly 
 contained in $V_{j,\ell^n}$, with the full multiplicity as in the 
 Chevalley-Weil formula (\ref{eq:chev_weil2}). That is, for the character $\chi_{V_{j,\ell^n}}$ of $V_{j,\ell^n}$, we have:
\[\chi_{V_{j,\ell^n}} = \bigoplus\limits_{i,s=0}^{\gcd(\ell^n,j)-1}
 \langle \psi_{\ell^n}, \chi_{ijs}\rangle_{H_{\ell^n}} \cdot \chi_{ijs}.\]

 From Lemma \ref{lemma:wolfart_14}, Proposition \ref{prop:module_decomp} and the discussion above, we obtain the following corollaries.

 \begin{corollary}
  The number $r$ of pairwise non-isogenous abelian varieties appearing in the decomposition of $\Jac(X_{\ell^n})$, which are not in the decomposition of $\Jac(Y_{\ell^n})$, is 
  \[ r= \sum\limits_{j=1}^{\ell^n-1} \gcd(\ell^n,j)^2 = \ell^{n-1}(\ell^n -1).\] 
 \end{corollary}

For the $\ell$-level Heisenberg curves $X_{\ell}$, write 
$V_j := V_{j,\ell}$. Additionally, note that the following 
$\ell>3$ assumption is based on Remark \ref{remark:ell_three}.

\begin{corollary}\label{cor:matching_decomps}
  For a prime $\ell >3$, we have mathching decompositions 
  \[H^0(X_{\ell},\Omega_{\overline{\Q}}) \cong H^0(Y_{\ell},\Omega_{\overline{\Q}}) \oplus 
\bigoplus\limits_{j=1}^{\ell-1} V_j, \quad \Jac(X_{\ell}) \sim \Jac(Y_{\ell}) \times \prod\limits_{j=1}^{\ell -1} \mathcal{A}_j^{k_j},\] 
meaning in the above decomposition of $H^0(X_{\ell},\Omega_{\overline{\Q}})$ as a $\overline{\Q}[H_{\ell}]$-module, 
the $V_j$ terms are isotypic components with characters $\chi_{V_j} = \frac{(\ell-3)}{2} \chi_{0,j,0}$.
\end{corollary}

\subsection{A Theorem by Kani and Rosen}\label{subsec:KaniRosen}

For this subsection let $\mathcal{C}$ denote a non-singular, projective, 
geometrically connected curve over a field of arbitrary characteristic. 
In \cite{KaniRosenIdempotent} Kani and Rosen studied idempotent relations in 
$\mathrm{End}(\Jac(\mathcal{C}))\otimes_\Z \Q$ and related them to factors of the Jacobian. 
Denote also by $\mathcal{C}/G$ the quotient curve for $G\subseteq \mathrm{Aut}(\mathcal{C})$. 
One of their theorems is the following.
\begin{theorem}[ Theorem B of \cite{KaniRosenIdempotent}]\label{thm:kani_rosen}
Let $G\subseteq \mathrm{Aut}(\mathcal{C})$ be a (finite) subgroup such that 
$G=A_1\cup \cdots \cup A_t$, where the subgroups $A_i\subseteq G$ satisfy $A_i\cap A_j = \{1\}$ for $i\neq j$. Then, we have the isogeny relation.

\[\Jac(\mathcal{C})^{t-1} \times \Jac(\mathcal{C}/G)^g \sim \Jac(\mathcal{C}/A_i)^{a_i} 
\times \cdots \times \Jac(\mathcal{C}/A_t)^{a_t},\]
where $g=|G|$ and $a_i=|A_i|$ for $i=1,\ldots,t$.
\end{theorem} 

Let us apply Theorem \ref{thm:kani_rosen} to the Heisenberg curve $X_{\ell}$. 
For this, we have to partition the discrete Heisenberg group $H_\ell$. Firstly, 
one way to achieve this is as follows. The group $H_{\ell}$ has order $\ell^3$ and 
every element in it has order $\ell$. Since two distinct groups of order $\ell$ must intersect only at
 $1$, we can partition $H_{\ell}$ with only subgroups of order $\ell$. 
 We have $\ell^3-1$ non-trivial elements fitting into parts with $\ell-1$ elements, thus $t= \ell^2+\ell+1$ for this partition, which yields
 \begin{equation}\label{eq:kani_rosen1}
 \Jac(X_{\ell})^{\ell+1} \sim \Jac(X_{\ell}/A_1) \times \cdots \times \Jac(X_{\ell}/A_{\ell^2+\ell+1}),
 \end{equation} and the $A_i$ are all the possible cyclic subgroups $\langle \beta^* c^* \alpha^*\rangle$, where by $*$ we denote arbritrary powers between $0$ and $\ell-1$.

We can partition $H_\ell$, in a second way, providing another isogeny relation for 
$\Jac(X_{\ell})$. Namely, denote by $A\subseteq H_{\ell}$ the subgroup 
$\langle \alpha \rangle \times \langle c \rangle \cong (Z/\ell\Z)^2$. It is a 
maximal abelian subgroup of $H_{\ell}$ which contains the center $Z(H_{\ell})$. We can now partition 
$H_\ell$ as $A$ and $H_{\ell}\setminus A$, where we also partition the latter as in the previous way, 
that is in terms of subgroups of order $\ell$ not contained in $A$. In this case, $A$ contains $\ell^2-1$ 
non-trivial elements, and we have to divide the rest $\ell^3 -\ell^2$ non-trivial elements into parts of $\ell-1$ elements. That is 
$t=1+\ell^2$, which yields the isogeny relation
\begin{equation}\label{eq:kani_rosen2}
\Jac(X_{\ell})^\ell \sim \Jac(X_{\ell}/A)^\ell \times \prod\limits_{A_i \not \subseteq A} \Jac(X_{\ell}/A_i). 
\end{equation} 
Combining now the isogeny relations (\ref{eq:kani_rosen1}) and (\ref{eq:kani_rosen2}), we get the isogeny relation

\[ \Jac(X_{\ell}) \times \Jac(X_{\ell}/A)^\ell \sim \prod\limits_{A_i \subseteq A} \Jac(X_{\ell/A_i}). \] Note that 
in the right-hand-side $\Jac(Y_{\ell})$ appears as a factor, since the quotient 
$X_{\ell}/Z(H_{\ell})=Y_{\ell}$ is the Fermat curve. It could, perhaps, be possible to use the above 
decomposition to say more about the simple abelian varieties appearing in Corollary \ref{cor:matching_decomps}, 
in terms of Jacobians of subcovers of $X_{\ell}\rightarrow \mathbb{P}^1$. However, as we will not require this 
information for our main Theorem \ref{thm:main}, we will not pursue this approach any further for now.

 \section{Complex Multiplication}\label{sec:cm}
In this section, we prove the main Theorem \ref{thm:main}, verifying Ihara's intuition from \cite{Ihara3point} 
that for $\ell>3$ the Jacobians of Heisenberg curves 
do not admit complex multiplication. We also discuss the $\ell=3$ case. We begin by recalling the precise definition of 
complex multiplication and some standard facts.

A simple polarized abelian variety $\mathcal{A}$ over a field of characteristic $0$ has 
{\em complex multiplication}, or in other words {\em is of CM-type}, if its endomorphism algebra 
\[ \mathrm{End}^0(\mathcal{A}) := \mathrm{End}(\mathcal{A})\otimes_\Z \Q \] is a number field 
$\mathbb{K}$ of degree $[\mathbb{K}:\Q] = 2\dim \mathcal{A}$. In this case, 
$\mathbb{K}$ is necessarily a CM-field, meaning it is a quadratic extension of a totally real field $F$ with 
$[F:\Q] = \dim \mathcal{A}$. If $\mathcal{A}$ is not simple, 
then it has a decomposition by the Poincaré complete reducibility theorem; 
we say $\mathcal{A}$ is of CM-type if every simple factor in the decomposition is of CM-type. 
This property is well-defined up to isogeny by Lemma \ref{lemma:isogenous}.
\begin{remark}
  The definition above is tailored to the characteristic $0$ case. Some authors prefer the phrase
  {\em admits sufficiently many complex multiplications}, in order to include both 0 and positive characteristic cases. 
  If $\mathcal{A}$ is an abelian variety over a field of positive characteristic, 
  one demands that the potentially non-commutative algebra $\mathrm{End}^0(\mathcal{A})$ 
  contains a commutative semi-simple subalgebra of rank $2\dim \mathcal{A}$ over $\Q$. 
  See \cite[4.2]{OortAbelianFiniteFields} for a 
  clarification regarding the above definition.
\end{remark}

\begin{lemma}\label{lemma:isogenous}
If $\mathcal{A} \sim \mathcal{B}$ are isogenous abelian varieties over 
a field of characteristic $0$, then 
$\mathcal{A}$ is of CM-type if and only if $\mathcal{B}$ is.
\end{lemma}
\begin{proof}
The lemma follows from the fact that isogenous abelian varieties have isomorphic endomorphism algebras, 
see e.g. \cite[pg. 43]{milneAV}.
\end{proof}
\begin{lemma}\label{lemma:subjac}
Let $X\rightarrow Y$ be a finite morphism of curves over $\mathbb{C}$. If $\Jac(X)$ is of CM-type, then $\Jac(Y)$ is of CM-type.
\end{lemma}
\begin{proof}
This follows from the fact that $\Jac(Y)$ is isogenous to an abelian 
subvariety of $\Jac(X)$, together with \ref{lemma:isogenous}.
\end{proof}
As before, we will rely on the framework developed by Wolfart, which relates the isogeny
decomposition of the Jacobian to the irreducible representations 
arising from the action of the automorphism group on the space of 
holomorphic differentials. Below, we state his main theorem from \cite{WolfartCMJacobians}.

\begin{theorem}[Wolfart's Theorem $8$, \cite{WolfartCMJacobians}]
    Let $\Gamma$ be a normal torsion-free subgroup of a Fuchsian 
    triangle group $\Delta$ and $X= \Gamma/ \mathbb{H}$ be the uniformized 
    Riemann surface of genus $g>1$ with many automorphisms. Let 
    $G=\Delta/\Gamma\subseteq \mathrm{Aut}(X)$ be the Galois group 
    of the Belyi function $X\rightarrow \Delta/\mathbb{H}$. Let $\Phi$ be the canonical representation of $G$ on $H^0(X,\Omega_{\overline{\mathbb{Q}}})$. For any irreducible subspace $U$ of $\Phi$ there exists a simple abelian variety $A_\nu$ of $\Jac(X)$ occuring with multiplicity $k_\nu$ and with endomorphism algebra $D_\nu$, with dimension $q_\nu^2 $ over its center, such that the following inequality holds.

    \[\frac{2q_\nu \dim_\mathbb{C} A_\nu}{\dim_{\mathbb{Q}}D_\nu} \leq 
    \dim_{\overline{\mathbb{Q}}} U \leq k_\nu q_\nu.\] In particular, for the endomorphism algebra it holds that
    \[2 \dim_\mathbb{C} A_\nu \leq k_\nu \dim_\mathbb{Q} D_\nu.\]
\end{theorem} 
We will specifically make use of the inequality
\begin{equation}\label{eq:ineq1}
\dim_{\overline{\Q}} U \leq k_\nu q_\nu, 
\end{equation}
thus we will pinpoint the exact argument of Wolfart proving it. This is a 
combination of \cite[Lemma 15]{WolfartCMJacobians} and the classification of simple abelian varieties and their endomorphism algebras 
\cite{AlbertA, MR156001}.

As in Lemma \ref{lemma:wolfart_14} we have subspaces 
$U_\nu \subseteq \Jac(X)$ that are invariant under the 
$G=: \Delta/\Gamma \subseteq \mathrm{Aut}(X)$ action. Each $U_\nu$ is 
isomorphic to $H^0(A_\nu^{k_\nu},\Omega_{\overline{\Q}}) \cong H^0(A_\nu,\Omega_{\overline{\Q}})^{k_\nu}$. 
Thus, if $U_{\nu\mu}, \mu=1,\ldots,m_{\nu}$ are the $D_\nu$-invariant irreducible subspaces of $H^0(A_\nu,\Omega_{\overline{\Q}})$, then 
$U_\nu = \sum_\mu (U_{\nu\mu})^{k_\nu}$. The action now of 
$G$ on $U_\nu$ is an action by endomorphisms 
$\mathrm{End}^0(\Jac(X))$, via the homomorphism 
$\Z[G]\rightarrow \mathrm{End}^0(\Jac(X))$, resctricted on $U_\nu$. As $\mathrm{End}^0(U_\nu) \cong \mathrm{End}^0(A_\nu^{k_\nu})\cong M_{k_\nu}(D_\nu)$, the action of $G$ on $U_{\nu}$ is via 
$k_\nu\times k_\nu$ matrices over $D_\nu$. Since each 
$U_{\nu\mu}$ is $D_\nu$-invariant, the entire block 
$(U_{\nu\mu})^{k_\nu}$ is $G$-invariant. Therefore, any $G$-irreducible subspace $U$ of $\Phi$ is contained in some $(U_{\nu\mu})^{k_\nu}$. This implies,
\[\dim_{\overline{\Q}} U \leq k_\nu \cdot \dim_{\overline{\Q}} U_{\nu\mu}.\] 

From the classification of simple abelian varieties and their endomorphism algebras \cite{AlbertA, MR156001}, each $D_\nu$ decomposes $H^0(A_\nu,\Omega_{\overline{\Q}})$ into irreducible subspaces of dimension 
$q_\nu$. That is $\dim_{\overline{\Q}} U_{\nu\mu } = q_\nu$, 
completing the justification of inequality (\ref{eq:ineq1}). 

We return our focus now on the Heisenberg curves $X_{\ell}$, 
for the primes $\ell > 3$. Recall that by $\mathcal{A}_j$ we denote the 
simple abelian varieties appearing in the decomposition of 
$\Jac(X_{\ell})$ from Corollary \ref{cor:matching_decomps}. 
To the abelian varieties $\mathcal{A}_j^{k_j}$ correspond the 
isotypic components $V_j$. The subspaces $V_j$ consist of 
$H_\ell$-irreducible spaces with characters $\chi_{0,j,0}$ of 
dimension $\ell$. Therefore, in this setting the inequality (\ref{eq:ineq1}) translates to
\begin{equation}\label{eq:ineq2}
\ell \leq k_j q_j,
\end{equation} where $q_j^2$ is the dimension of the endomorphism 
algebra $D_j := \mathrm{End}^0(\mathcal{A}_j)$ over its center.

We now prove the following lemmas and then we combine all the ingredients to state and prove our main theorem.
\begin{lemma}\label{lemma:ineq_dj}
The dimensions of the endomorphism algebras $D_j$ of each simple abelian variety $\mathcal{A}_j$ appearing in the decomposition of $\Jac(X_{\ell})$ satisfy the following inequality:
\[ \dim_\Q D_j \geq \ell-1. \]
\end{lemma}
\begin{proof}
The central generator $c=[\alpha,\beta]$ of order $\ell$ in $H_\ell$ 
acts diagonally by $\zeta^j$ on $V_j$. By duality, since 
$\Jac(X_\ell) \cong H^0(X_\ell,\Omega_{\mathbb{C}})^*/H_1(X_\ell,\mathbb{Z})$ 
when viewed over $\mathbb{C}$, the element $c$ induces an endomorphism 
on $\mathcal{A}_j^{k_j}$ that acts as scalar multiplication by 
$\zeta^{-j}$. Here we use the fact that we can extend our scalars to $\mathbb{C}$ and then descend to $\overline{\Q}$, since $X_{\ell}$ 
is defined over $\overline{\Q}$. 
Thus, we have a non-trivial algebra homomorphism
\[\mathbb{Q}[\langle c \rangle] \longrightarrow D_j,\] which is 
necessarily injective since $ \mathbb{Q}[\langle c \rangle] \cong \mathbb{Q}(\zeta)$. Therefore,
\[\dim_\mathbb{Q} D_j = [D_j : \mathbb{Q}(\zeta)]\cdot  
[\mathbb{Q}(\zeta): \mathbb{Q} ] =  [D_j : \mathbb{Q}(\zeta)] \cdot (\ell-1),\] from which the result follows.
\end{proof}
\begin{lemma}\label{lemma:inv_system}
The Heisenberg curves $X_{\ell^n}$ for $n\geq 1$ form an inverse system 
of branched covers of $\mathbb{P}^1_{\overline{\Q}}$ which are étale 
outside $0,1,\infty$.
\end{lemma}
\begin{proof}
Recall that $\overline{\Q}(t)$ is the projective $t$-line and set $\zeta_{\ell^n}$ a primitive $\ell^n$-root of unity, as well as 
$x_n = t^{1/\ell^n}, y_n = (1-t)^{\ell^n}$ and 
$\varepsilon_n$ satisfying 
\[ \varepsilon_n^{\ell^n} = \prod\limits_{i=1}^{\ell^n -1} 
(1-\zeta_{\ell^n} x_n)^i. \] Then $\mathcal{K}_{\ell^n}=\overline{\Q}(x_n,y_n)$ 
is the function field of the Fermat curve $Y_{\ell^n}$ and it is obvious 
that $\mathcal{K}_{\ell^n} \subseteq \mathcal{K}_{\ell^{n+1}}$. 
Then, $\mathcal{R}_{\ell^n} = \mathcal{K}_{\ell^n}(\varepsilon_n)$ is the 
function field of the Heisenberg curve $X_{\ell^n}$. Observe that
\begin{align*}
\varepsilon^{\ell^{n+1}}_{n+1} & = 
\prod\limits_{i=1}^{\ell^{n+1}-1}( 1 - \zeta_{\ell^{n+1}}^i x_{n+1})^i \\ 
&= \prod\limits_{i=0}^{\ell^n-1} \left( \prod\limits_{\lambda=0}^{\ell-1} ( 1 - \zeta_{\ell^{n+1}}^{i+\lambda \ell^n} x_{n+1}) \right)^i \cdot 
\prod\limits_{i=0}^{\ell^n-1}\prod\limits_{\lambda=0}^{\ell-1} ( 1 - \zeta_{\ell^{n+1}}^{i+\lambda \ell^n} x_{n+1})^{\lambda \ell^n} \\
&= \varepsilon_{n}^{\ell^n} \cdot w^{\ell^n}, 
\end{align*}
where $w$ is an element in $\mathcal{K}_{\ell^{n+1}}$. Therefore 
$\varepsilon_{n+1}^{\ell} = \varepsilon_n \cdot w$ which implies 
$\mathcal{R}_{\ell^n}\subseteq \mathcal{R}_{\ell^{n+1}}$. Now from the 
anti-equivalence between function fields of transendence degree $1$ and non-singular projective 
curves, both over $\overline{\Q}$, as well as the tower of function fields 
$\overline{\Q}(t)\subseteq \mathcal{R}_\ell \subseteq \mathcal{R}_{\ell^2} \subseteq \mathcal{R}_{\ell^3} 
\subseteq \cdots $, we have the inverse system of morphisms 
$X_{\ell^m}\rightarrow X_{\ell^n}$ for $m \geq n \geq 1$, which 
are also compatible as covers of $\mathbb{P}^1_{\overline{\Q}}-\{0,1,\infty\}$ if the appropriate preimages are removed.
\end{proof}
Alternatively, this fact can also be seen via the pro-$\ell$ étale fundamental group of 
$\mathbb{P}^1_{\overline{\Q}}-\{0,1,\infty\}$. Denote by $\pi_1$ and $\pi_1^{\et}(\ell)$ its 
topological fundamental group (over $\mathbb{C}$ with the complex topology ) 
and its pro-$\ell$ étale fundamental group respectively. For any choice of coordinates 
$x,y,z$ as homotopy classes of loops around the branch points ${0,1,\infty}$, we have that 
$\pi_1 = \langle x,y,z | xyz=1 \rangle$ and $\pi_1^{\et}(\ell)$ is its pro-$\ell$ completion. 
We have the kernel $N_n$ of $\pi_1\rightarrow H_{\ell^n}$, where $N_n$ is 
normally generated by the elements 
$x^{\ell^n},y^{\ell^n}, [x,[x,y]], [y,[x,y]]$. Then, the pro-$\ell$ completion 
$\hat{N}_n$ is the kernel of 
$\pi_1^{\et}(\ell)\rightarrow H_{\ell^n}$. Now $\hat{N}_n$ is the pro-$\ell$ étale 
fundamental group of the Heisenberg curve $X_{\ell^n}$ minus the points lying above 
$0,1,\infty$. It is easy to see that $N_{n+1}\subset N_{n}$, which yields 
$\hat{N}_{n+1}\subset \hat{N}_{n}$ by left exactness of the profinite functor, 
since $N_{n+1}$ is of finite index in $N_n$. Thus, we 
get the morphisms $X_{\ell^m}\rightarrow X_{\ell^n}, m\geq n$ by Galois correspondence.

\begin{theorem}\label{thm:main}
Let $X_{\ell^n}$ be the $\ell^n$-level Heisenberg curve, for a prime 
$\ell > 3$. Then $\Jac(X_{\ell^n})$ is \textbf{not} of CM-type.
\end{theorem}
\begin{proof}
By the combination of Lemmas \ref{lemma:isogenous}, 
\ref{lemma:subjac} and \ref{lemma:inv_system}, it suffices to prove the non CM-type only for the base case $n=1$.

We will thus prove it for $\Jac(X_{\ell})$. It is well-known that Fermat Jacobians are of CM-type \cite{MR511556}, 
so we have to work with the simple abelian varieties $\mathcal{A}_j$ from Corollary 
\ref{cor:matching_decomps}. Assume the contrary that $\Jac(X_{\ell})$ is of CM-type, 
which means every $\mathcal{A}_j$ is by Lemma \ref{lemma:isogenous}. That is the endomorphism 
algebra $D_j$ of $\mathcal{A}_j$ is a CM-field of dimension $2 \dim_\mathbb{C} \mathcal{A}_j$ 
over $\Q$. Consequently, it has dimension $q_j=1$ over its center. Thus, inequality (\ref{eq:ineq2}) becomes 
\[\ell \leq k_j, \quad j=1,2,\ldots,\ell-1,\] and furthermore, we have 
\[\dim_{\mathbb{C}}\mathcal{A}_j^{k_j} = k_j \dim_{\mathbb{C}}\mathcal{A}_j = 
\dim_{\overline{\Q}} V_j = \frac{\ell(\ell-3)}{2}, \quad \ell > 3, \] as the dimensions are carried over in the construction of invariant subspaces from Lemma \ref{lemma:wolfart_14}. By the CM-type assumption, we have that $2\dim_\mathbb{C} \mathcal{A}_j = \dim_\Q D_j$ and thus by Lemma \ref{lemma:ineq_dj} 
\[ \frac{(\ell -1)}{2} \leq \dim_{\mathbb{C}} \mathcal{A}_j. \] Combining everything, we have
\[ \frac{(\ell -1)}{2} \leq \dim_{\mathbb{C}} \mathcal{A}_j \leq \frac{(\ell-3)}{2},\] which is clearly a contradiction. 
In fact, we have proven that none of the abelian varieties $\mathcal{A}_j$ is of CM-type, and the theorem follows.
\end{proof}

The takeaway of this theorem is that, even though the Heisenberg curves $X_{\ell^n}$ are curves 
with many automorphisms, the representation theory of the automorphism subgroup $H_{\ell^n}$ dictates the decomposition of 
$\Jac(X_{\ell^n})$ into simple blocks whose endomorphism algebras contain the CM-field $\Q(\zeta)$, 
but are not CM-fields themselves. The notion of simple abelian varieties over $\mathbb{C}$ whose endomorphism algebra contains a CM-field appears as {\em generalized complex multiplication} in the work of Shiga and Wolfart \cite{WolfartShiga}.
%
%of dimension large enough to admit 
%a $\Q(\zeta)$ action, but not large enough to be able to admit a CM action with respect to their entire 
%dimension. This phenomenon appears as {\em generalized complex multiplication} in the work of Shiga and Wolfart \cite{WolfartShiga}.

\subsection{The \texorpdfstring{$\ell=3$}{} case.} \label{subsec:ell_three}
By Remark \ref{remark:ell_three} the abelian varieties $\mathcal{A}_j$ do not exist for $\ell=3$. 
In fact, the Heisenberg curve $X_3$ is an elliptic curve isogenous to the Fermat curve 
$Y_3$, as discussed in \S 2, which implies $\Jac(X_3)$ is of CM-type. 

So it is a question whether CM-type Jacobians are to be found on the rest of the $\ell$-tower. We can 
replicate the previous arguments specifically for $X_9$ to show that actually $\Jac(X_9)$ 
is not of CM-type. Indeed, denote by $\mathcal{A}$ a simple abelian variety from its 
decomposition, with endomorphism algebra $D$ and multiplicity $k$. We have that 
$D$ contains $\Q(\zeta_9)$, for $\zeta_9$ a primitive $9$-th root of unity, which is of dimension $6$ over $\Q$. Thus $6 \leq \dim_\Q D$. Suppose that $\mathcal{A}^k$ corresponds to the isotypic component of the 
irreducible representation with corresponding character $\chi_{1,3,1}$, i.e. $i=s=1$ and $j=3$, from the 
space of holomorphic differentials. The formula of Chevalley-Weil (\ref{eq:chev_weil2}) in this case yields the multiplicty $2$. 
Since the character $\chi_{1,3,1}$ is of dimension 
$3$, we have $\dim_{\mathbb{C}} \mathcal{A}^k = 6$ and Wolfart's inequality (\ref{eq:ineq1}) yields $3\leq kq$, 
for $q^2$ the dimension of $D$ over its center. These are enough to prove that $\mathcal{A}$ is not of CM-type, since if we are to 
assume otherwise we get the contradiction $3 \leq \dim_{\mathbb{C}} \mathcal{A} \leq 2$. This finishes the claim that $\Jac(X_9)$ is not of CM-type and we have proved the following complementary theorem.
\begin{theorem}\label{thm:complement}
Let $X_{3^n}$ be the $3^n$-level Heisenberg curve, for an integer $n\geq2$. Then $\Jac(X_{3^n})$ is \textbf{not} of CM-type.
\end{theorem}
It is worth noting that, the previous argument does not work if we use the character 
$\chi_{0,1,0}$ for $j=1$ and $i=s=0$, which is of dimension $9$. In this case, 
the two inequalities would yield $\dim_\mathbb{C}\mathcal{A} = 3$ with no contradiction, 
and the corresponding simple abelian variety $\mathcal{A}$, which is not isogenous to a 
subvariety of $\Jac(Y_9)$ (these correspond to one-dimensional characters and $\chi_{0,1,0}$ is $9$-dimensional), could 
potentially be of CM-type with endomorphism algebra $\Q(\zeta_9)$.

\subsection{A question by Oort.}\label{subsec:oort}

As it appears in the literature, specifically in \cite{OortCMJacobians} and \cite[Question 2.4]{ObusShaska}, 
Oort initially hoped to be able to construct CM Jacobians from curves with many automorphisms. 
As it turned out, this is not always possible. Wolfart \cite{WolfartCMJacobians} 
classified the entire case genus $\leq 4$ case and showed there are counter-examples to Oort's question. 
Specifically, he found one counter example in genus $3$, which is a hyperelliptic curve, 
and two counterexamples in genus $4$. Wolfart also 
shows McBeath's curve \cite[Section 6.5]{WolfartCMJacobians}, which is of genus $7$, does not have a CM-type Jacobian.

Furthermore, more counter-examples were given by M\"uller and Pink in the case of hyperelliptic curves \cite{PinkMueller}. 
They classified them into $3$ infinite families with CM-type Jacobians and $15$ exceptional curves; $10$
out of which do not have CM-type Jacobians. The largest genus of their counterexamples is $30$.

Later in \cite{ObusShaska}, Obus and 
Shaska generalized this result by classifying all the superelliptic curves with many automorphisms. 
They found $68$ exceptional superelliptic curves, out of which $33$ are without CM Jacobians (including the $10$ hyperelliptic ones), 
with the largest genus being $1830$.

In all these cases the counterexamples so far appeared to be sporadic instances, 
as exceptional curves not belonging in one out of three infinite families, 
and their genus remained bounded. In the present paper, we have shown that this phenomenon is not contained 
in isolated cases. From Theorems \ref{thm:main} and 
\ref{thm:complement} we provide an infinite family of curves with many automorphisms of 
unbounded genus, as in equation (\ref{eq:genus}), with non-CM Jacobians. 
That is Heisenberg curves are curves with many 
automorphisms, by Remark \ref{remark:many_aut}, as they are normal Belyi covers.
\subsection{Remarks on Streit's CM criterion} \label{subsec:streit}

In their studies of superelliptic and hyperelliptic curves \cite{PinkMueller,ObusShaska}, Pink, Müller, Obus and Shaska
utilized the following criterion to verify the CM-type Jacobians. Let $\mathcal{C}$ 
be a curve with many automorphisms, and $G\subseteq \Aut(\mathcal{C})$ act on $H^0(\mathcal{C},\Omega_{\mathbb{C}})$ with character 
$\chi$. Let
$\chi_{\textrm{triv}}$ denote the trivial character of $G$ and let $\mathrm{Sym}^2 \chi$ be the character of the $G$-action 
on the symmetric square $\mathrm{Sym}^2H^0(\mathcal{C},\Omega_\mathbb{C})$. Then Streit's 
CM criterion \cite[Proposition 5]{StreitPeriods} is the following: If 
\[ \langle \mathrm{Sym}^2 \chi, \chi_{\textrm{triv}} \rangle_G = 0,\] then 
$\Jac(\mathcal{C})$ is of CM-type. In the same proposition, Streit provides an 
equivalent statement in terms of real, complex, and quaternionic irreducible characters.
A geometric intuition of this criterion is also stated in \cite{PinkMueller} by noting that 
$\mathrm{Sym}^2H^0(\mathcal{C},\Omega_\mathbb{C})$ is naturally isomorphic to the tangent space of 
the Siegel moduli space for the genus of $\mathcal{C}$. In practice, a useful way to compute the character 
$\mathrm{Sym}^2 \chi$ is the following 
\[ \mathrm{Sym}^2 \chi(g) = \frac{1}{2}\left(\chi(g^2) + \chi(g)^2\right).\]
Furthermore, the exhaustive classification in the hyperelliptic and superelliptic cases showed that for these curves,
 the converse of Streit's criterion also holds. Obus and Shaska consequently pose the question: what is the geometric intuition for this phenomenon?

We may add more context to the question, since the Heisenberg curves $X_{\ell^n}$--which we have proven do not have 
Jacobians of CM-type--also satisfy a similar converse to Streit's criterion. Indeed, notice that all the characters $\chi_{ijs}$ are complex 
(i.e., having Schur indicator 0), except for the trivial character $\chi_{0,0,0}$. By comparing the multiplicities given by the
Chevalley-Weil formula (\ref{eq:chev_weil2}) with the character $\chi_{\ell^n}$ from 
equation (\ref{eq:chi_elln}) on the first (co)homology group, we observe that (except for the curve $X_3$) the 
complex characters are not disjoint between the spaces of holomorphic and anti-holomorphic differentials. Thus, they fail 
to satisfy Streit's criterion \cite[Proposition 5]{StreitPeriods}. Note that we do not use the full automorphism group 
of $X_{\ell^n}$ in this argument, but this is justified by \cite[Remark 4.5]{ObusShaska}.

\section{Ihara Theory}\label{sec:ihara}

Grothendieck's 1984 {\em esquisse d'un programme} \cite{MR1483107} proposed the study of the outer action of 
$G_\Q := \Gal(\overline{\Q}/\Q)$ on the profinite completion of the fundamental group of 
$\mathbb{P}^1-\{0,1,\infty\}$. Based on this in the following years, in a series of seminal papers 
\cite{Ihara1986towers,Ihara1985-it,IharaTateTwist, 
MR1159208}, Ihara initiated the study of the Galois action on the pro-$\ell$ representation 
\[ \phi : G_\Q \longrightarrow \mathrm{Out}(\mathcal{F}), \] where $\mathrm{Out}$ denotes the outer automorphisms and $\mathcal{F}$ is the inverse limit of all 
$\ell$-power order quotients of the topological fundamental group $\pi_1(\mathbb{P}^1_\mathbb{C} - \{0,1,\infty\}, b) 
= \langle x,y,z \mid xyz=1\rangle $. This action is independent of the basepoint $b$. Regarding $\phi$, 
the natural questions one may ask, and Ihara focused on, are what is the image and the kernel? 
On the one hand, the study of $\mathrm{im}\phi$ led to the term {\em profinite braids}:
\[ \{ \sigma \in \mathrm{Aut}(\mathcal{F}): \sigma(x) \sim x^{\alpha}, \ \sigma(y) \sim y^{\alpha}, \ \sigma(z)\sim z^{\alpha}, 
\ \alpha \in \Z_{\ell}^* \}, \] where $\sim$ denotes conjugation. The above group can be seen as a generalization of 
Artin's presentation of 
the braid group as a mapping class group of a punctured closed disk \cite{FarbMagalit}. A fact which inspires 
Ihara's name in arithmetic topology \cite{MorishitaATIT,MR3912937}, see also \cite{KontogarParamFermat}. On the other hand, 
although in the profinite case the kernel is trivial as a consequence of Belyi's theorem, 
in the pro-$\ell$ case it is not trivial, and to study it, Ihara defined a descending filtration $\{ G_{\Q_{(m)}}\}_{m\geq 1}$ 
of $G_\Q$ such that $\cap_m G_{\Q_{(m)}} = \ker \phi$, which is given by the kernels of 
\[\phi_m : G_\Q \longrightarrow \mathrm{Out}(\mathcal{F}/ \mathcal{F}(m+1)),\] where $\mathcal{F}(1)=\mathcal{F}$ and 
$\mathcal{F}(m+1)=[\mathcal{F},\mathcal{F}(m)]$ is the lower central series of $\mathcal{F}$. These subgroups are closed in the 
pro-$\ell$ topology of $\mathcal{F}$ and characteristic so that the $\phi_m$ are well-defined. This filtration is of 
utmost importance, since based on $\mathrm{gr}^m G_\Q := \Gal(\Q_{(m+1)}/\Q_{(m)})$, he defined the graded object 
\[\mathfrak{g}^\ell = \bigoplus\limits_{m\geq 1} \mathrm{gr}^m G_\Q,\] and along with Deligne, conjectured that the graded Lie algebra $\mathfrak{g}^\ell \otimes_{\Z_\ell} \Q_\ell$ is freely 
generated by the Soulé characters. For more details and their definition see \cite{SharifiSouleIhara}, and \cite{BrownSurveyDeligneIhara}. Additionally, in joint paper with Anderson in 1988 \cite{AndersonIhara88}, Ihara fully 
determined 
$\ker\phi$ as a group, which lead to another question. Is its corresponding field the maximal extension of 
$\Q(\mu_{\ell}^\infty) := \cup_n \Q(\zeta_{\ell^n})$ that is unramified outside $\ell$? 
The current state of the art is that the answer is 
positive for the regular primes $\ell$, a result intimately involving the Vandiver conjecture. The work of Sharifi 
\cite{MR1935409} proves this as a
consequence of the Deligne-Ihara conjecture, which was subsequently proven by Brown in \cite{MR2993755}.

Before this theory was as understood as it is today, and even before 
explicitly determing $\ker\phi$, Ihara compared 
the profinite braids and the filtration $\Q_{(m)}$, both defined with respect to 
$\mathcal{F}$, with alternative definitions based on $\mathcal{F}^*:= \mathcal{F}/[\mathcal{N},\mathcal{N}]$, for $\mathcal{N}$ some 
normal closed subgroup of $\mathcal{F}$. The case of the Fermat tower of curves $Y_{\ell^n}$ is particularly interesting in this setting, 
as these dominate the abelian $\ell$-power coverings of $\mathbb{P}^1-\{0,1,\infty\}$ and these correspond to 
$\mathcal{N}=\mathcal{F}(2)$, in the towers of curves viewpoint of \cite{Ihara1986towers}. To bring complex multiplication into the picture, the group 
$\mathcal{N}/[\mathcal{N},\mathcal{N}]$ can be canonically identified with 
$\varprojlim T_{\ell}(\Jac(Y_{\ell^n}))$, which is the inverse limit of the Tate modules of their respective Jacobians.

Set $\mathcal{F}^* := \mathcal{F}/[\mathcal{F}(2),\mathcal{F}(2)]$ and as previously define the lower central series $\mathcal{F}^*(m)$. We have the homomorphisms 
\[\psi : G_\Q\longrightarrow \mathrm{Out}(\mathcal{F}^*),\] and 
\[\psi_m : G_\Q\longrightarrow \mathrm{Out}(\mathcal{F}^*/\mathcal{F}^*(m+1)),\]
with fields $\Q_{[m]}$ corresponding to the kernels $\ker\psi_m$. Following \cite{Ihara3point} strictly, Ihara compared the 
$1$-profinite braids (of exponent $\alpha=1$) 
of $\mathcal{F}$ and $\mathcal{F}^*$ to establish an equivalence
\[\ker \phi = \ker\psi \iff \Q_{[m]}=\Q_{(m)}, \textrm{ for all } m\geq 1, \iff \mathrm{rank}\ \mathrm{gr}^m G_\Q = c_m^\prime, \textrm{ for all } m\geq 1\]
where $c_m^\prime = 1$ if $m$ is odd and $0$ otherwise, for $m\geq 3$. His focus was to use curves with 
non-CM Jacobians in order to show that the field corresponding 
to $\ker\phi$ is a non-abelian extension of
 $\Q(\mu_{\ell^\infty})$. In contrast, the fixed field 
 $\ker\psi$ yields an abelian extension, since each intermediate extension 
 $\Q_{[m]}/\Q(\mu_{\ell^\infty})$ is abelian. In turn, this 
 would show that the two filtrations do not coincide. 

 The last paragraph of \cite{Ihara3point} reads: ``For $\ell>3$ one might use Heisenberg curves instead, but at present, the 
 author does not know whether their Jacobians do not really have enough CM''. Thus, if $K_\phi$ is the 
 field corresponding to $\ker\phi$, our Theorem \ref{thm:main} implies that the extension $K_\phi/\Q(\mu_{\ell^\infty})$ 
 is non-abelian. This non-abelian fact was, of course, already understood by Anderson and Ihara in \cite{AndersonIhara88}. 
 Their method involved ``generating'' a large family of 
 curves (including Heisenberg curves) as products of curves of genus $0$, which would be used to determine $K_\phi$. However, their method did not explicitly 
 focus on the previous question of Ihara regarding the Heisenberg Jacobians. As it appears in the literature, this 
 question remained unresolved until the 
 present paper.

\bibliography{AKGeneral.bib}
\bibliographystyle{plain}
\end{document}